\documentclass[a4paper,10pt]{amsart}
\usepackage{amssymb,amsmath,amsfonts,amsthm}
\usepackage[dvips]{graphicx}
\usepackage{psfrag}

\theoremstyle{plain}
\newtheorem{main}{Theorem}

\newtheorem{theorem}{Theorem}[section]
\newtheorem{lemma}[theorem]{Lemma}
\newtheorem{proposition}[theorem]{Proposition}
\newtheorem{corollary}[theorem]{Corollary}
\theoremstyle{remark}
\newtheorem{remark}[theorem]{Remark}

\newcommand{\Leb}{\operatorname{vol}}

\newcommand{\C}{\operatorname{C}}

\newcommand{\SPH}{\operatorname{SPH}}
\newcommand{\Gibbs}{\operatorname{Gibbs}}

\newcounter{property}

\newenvironment{property*}[1][]{\par\medskip
\noindent\textbf{P\theproperty#1'.} \rmfamily}{\medskip}

           \def\ea{\end{array}}
          \def\ec{\end{center}}
     \def\ed{\end{description}}
        \def\ee{\end{equation}}
       \def\eea{\end{eqnarray}}
     \def\eeaa{\end{eqnarray*}}
 \def\et{\end{thebibliography}}

\def\Orb{{\rm Orb}}
\def\Diff{{\rm Diff}}

\def\Gibb{{\rm Gibb}}

\def\supp{\operatorname{supp}}

\def\cA{{\mathcal A}}

\def\cD{{\mathcal D}}

\def\cU{{\mathcal U}}
\def\cV{{\mathcal V}}

\def\cR{{\mathcal R}}

\def\cF{{\mathcal F}}
\def\cM{{\mathcal M}}

\def\cP{{\mathcal P}}
\def\cQ{{\mathcal Q}}
\def\cR{{\mathcal R}}

\def\cW{{\mathcal W}}

\def\loc{\operatorname{loc}}

\def\epsilon{\varepsilon}

\def\TT{{\mathbb T}}
\def\RR{{\mathbb R}}

\def\ZZ{{\mathbb Z}}

\def\MM{\operatorname{MM}}
\def\Int{\operatorname{int}}

\title[Measures of maximal $u$-entropy]
      {Thermodynamical $u$-formalism I:\\ measures of maximal $u$-entropy for maps that factor over Anosov}
\author{Raul Ures, Marcelo Viana, Fan Yang and Jiagang Yang}
\date{\today}

\thanks{M.V. and J.Y. were partially supported by CNPq, FAPERJ, and PRONEX. R. U. was partially supported by NSFC 11871262.
We acknowledge support from the Fondation Louis D--Institut de France (project coordinated by M. Viana).}

\address{Department of Mathematics, Southern University of Science and Technology, 1088 Xueyuan Rd., Xili, Nanshan District, Shenzhen, Guangdong, China 518055}
\email{ures\@@sustc.edu.cn}

\address{IMPA, Est. D. Castorina 110, 22460-320 Rio de Janeiro, Brazil}
\email{viana\@@impa.br}

\address{Department of Mathematics, Michigan State University, MI, USA.}
\email{yangfa31\@@msu.edu}

\address{Departamento de Geometria, Instituto de Matem\'atica e Estat\'\i stica, Universidade Federal Fluminense, Niter\'oi, Brazil}
\email{yangjg\@@impa.br}

\begin{document}

\maketitle

\begin{abstract}
We construct measures of maximal $u$-entropy for any partially hyperbolic diffeomorphism that factors
over an Anosov torus automorphism and has mostly contracting center direction.
The space of such measures has finite dimension, and its extreme points are ergodic measures with
pairwise disjoint supports.
\end{abstract}


\section{Introduction}
The \emph{variational principle}, proved in the early 1970s by
Dinaburg~\cite{Din70,Din71}, Goodman~\cite{Gdm71} and
Goodwin~\cite{Gdw71}, asserts that the topological entropy $h(f)$ of a
continuous transformation $f:M\to M$ in a compact metric space $M$
coincides with the supremum of the entropies $h_\mu(f)$ of its invariant
probability measures.
In general, the supremum is not attained but \emph{measures of maximal
entropy} do exist in some important situations.
That includes, for instance, expansive maps, as in that case the
entropy $h_\mu(f)$ is an upper semi-continuous function of the
probability measure $\mu$.

A more general version of the variational principle, due to Walters~\cite{Wal75},
states that, for any continuous function  $\phi:M\to \RR$, the (topological)
pressure $P(f,\phi)$ coincides with supremum of
$$
h_\mu(f) + \int_M \phi \, d\mu
$$
over all probability measures $\mu$. The situation in the previous
paragraph corresponds to the case $\phi\equiv 0$.
In general, when they exist, measures that realize the supremum are
called \emph{equilibrium states} of $\phi$.
The semi-continuity argument to prove existence for expansive
transformations remains valid in this generality.
See \cite{FET16} for a detailed review of these classical facts.

Still in the 1970s, Sinai~\cite{Sin72}, Ruelle~\cite{Rue76b}, and
Bowen~\cite{BR75,Bow75a} brought in ideas from statistical mechanics
to give an explicit construction of equilibrium states for a broad
class of transformations, including expanding maps and uniformly
hyperbolic diffeomorphisms. In addition to proving existence and
uniqueness, their theory leads to detailed information on the properties
of such measures, including mixing properties, decay of correlations
and large deviations principles.

When the transformation admits an invariant structure, such as an
invariant lamination or even a foliation, one may consider the
corresponding \emph{partial entropy}, which measures the complexity
of the dynamics along such a structure. This idea can be traced back
to Pesin~\cite{Pes77} and, more explicitly, Ledrappier~\cite{Led84a}
and Ledrappier, Young~\cite{LeY85a,LeY85b}.
In these papers, the  relevant invariant structure is the Pesin
unstable lamination.

Another important case, which concerns us more directly, is that of
the strong-unstable foliation $\cF^{uu}$ of a partially hyperbolic
diffeomorphism $f:M\to M$. We denote by $h_\mu(f,\cF^{uu})$ the
corresponding partial entropy of each invariant measure $\mu$,
and we call it \emph{$u$-entropy} of $\mu$.
In this setting, the notion of \emph{partial topological entropy} was
defined by Saghin, Xia~\cite{SaX09}. We call it the
\emph{topological $u$-entropy}
of the diffeomorphism, and denote it as $h(f,\cF^{uu})$.

Hu, Hua, Wu~\cite{HHW17} proved the variational $u$-principle
-- $h(f,\cF^{uu})$ coincides with the supremum of $h_\mu(f,\cF^{uu})$
over all $\mu$ -- and they also proved that the supremum is attained.
Recently, Hu, Wu, Zhu~\cite{HWZ} introduced the notion of
\emph{$u$-pressure} of a continuous function $\phi:M\to\RR$ and
extended the previous results to that setting: the $u$-pressure is equal
to the supremum of
\begin{equation}\label{eq.HHW}
h_\mu(f,\cF^{uu}) + \int_M \phi \, d\mu
\end{equation}
over all invariant probability measures $\mu$.

In this paper we initiate a program to extend the classical thermodynamical formalism
of Sinai, Ruelle and Bowen, to the language of $u$-entropy and $u$-pressure.
An invariant probability measure that realizes the supremum in \eqref{eq.HHW} is
called an \emph{equilibrium $u$-state}.
A long term goal is to prove that generic partially hyperbolic diffeomorphisms admit
finitely many ergodic equilibrium $u$-states, for any given H\"older potential $\phi$.
Moreover, these equilibrium states should exhibit fast decay of correlations and other
mixing-type properties.

It should be noted that in the setting of uniformly hyperbolic systems, the notions of
equilibrium state and equilibrium $u$-state coincide.
Our results (see  Section~\ref{s.example5}) show that this is no longer the case, in
general, for partially hyperbolic systems, the reason being that the (total) entropy
is vulnerable to local phenomena, such as small horseshoes, while the $u$-entropy is not.
As a consequence, in the broader partially hyperbolic setting equilibrium states may
fail to reflect the global dynamics, a role which is kept by equilibrium $u$-states.

Another point worth making is that the theory we deal with here reveals a close relation
between equilibrium $u$-states and the notion of Gibbs $u$-state
(see Theorem~\ref{main.maximal.measures} below).
The latter was introduced by Pesin, Sinai~\cite{PeS82} in the 1980s, and has been used
extensively (see~\cite[Chapter~11]{Beyond}) but, to the best of our knowledge, had not
previously been linked to the variational principle.

Presently, we focus on a class of partially hyperbolic diffeomorphisms that factor over an
Anosov torus automorphism. Let us begin with an informal presentation of the main
notions, postponing most formal definitions to Section~\ref{s.definitions}.

We take $f:M\to M$ to be a partially hyperbolic, dynamically coherent $C^1$ diffeomorphism
on a compact manifold $M$, with partially hyperbolic splitting $TM = E^{cs} \oplus E^{uu}$.
We say that $f$ \emph{factors over Anosov} if there exist a hyperbolic linear automorphism
$A:\TT^d\to\TT^d$ on some torus, and a continuous surjective map $\pi:M\to\TT^d$ such that
\begin{enumerate}
    \item[(H1)] $\pi \circ f = A \circ \pi$;
    \item[(H2)] $\pi$ maps each strong-unstable leaf of $f$ homeomorphically to an unstable leaf of $A$;
    \item[(H3)] $\pi$ maps each center-stable leaf of $f$ to a stable leaf of $A$.
\end{enumerate}
There are many examples of such maps, some of which we review in Sections~\ref{s.example1} through~\ref{s.example5},
including diffeomorphisms derived from Anosov, partially hyperbolic skew-products, and certain topological solenoids.
Condition (H2) implies that $E^{uu}$ has the same dimension as the unstable subbundle of $A$;
the dimension of $E^{cs}$ may be greater than that of the stable subbundle of $A$.

By pulling the Lebesgue measure along the unstable leaves of $A$ back under the factor map $\pi$,
one obtains a special family of measures on the strong-unstable plaques of $f$ that we call the
\emph{reference measures} (or \emph{Margulis measures}, see~\cite{Mar70}).
We say that $f$ has \emph{$c$-mostly contracting center} if
\begin{equation}\label{eq.c_mostly_contracting}
\limsup_n \frac 1n \log \|Df^n \mid_{E^{cs}}\| < 0
\end{equation}
on a positive measure subset relative to every reference measure. This is similar to the notion of
\emph{mostly contracting center direction} introduced by Bonatti, Viana~\cite{BoV00}, except that
we replace the Lebesgue measure along the strong-unstable leaves of $f$ with our reference measures
in the definition. The motivation behind our terminology will be explained in Remark~\ref{r.terminology}.

The \emph{topological $u$-entropy} of $f$, denoted by $h(f,\cF^{uu})$, is the maximal rate of volume growth
for any disk contained in an strong-unstable leaf. See Saghin, Xia~\cite{SaX09}.
The \emph{$u$-entropy} of an $f$-invariant measure $\mu$, denoted as $h_\mu(f,\cF^{uu})$,
is defined by
$$
h_\mu(f,\mu) = H_\mu\left(f^{-1} \xi^u \mid \xi^u\right)
$$
where $\xi^u$ is any measurable partition subordinate to the strong-unstable foliation
(see the Appendix for the definition).
Recall that, according to Rokhlin~\cite[Section~7]{Rok67a}, the entropy $h_\mu(f)$ is the
supremum of $H_\mu\left(f^{-1} \xi \mid \xi\right)$ over all measurable partitions $\xi$
with $f^{-1}\xi \prec \xi$. Thus we always have
\begin{equation}\label{eq.two_entropies}
h_\mu(f,\cF^{uu}) \le h_\mu(f).
\end{equation}
See Ledrappier, Strelcyn~\cite{LeS82}, Ledrappier~\cite{Led84a},  Ledrappier, Young~\cite{LeY85a},
and Yang~\cite{Yan16}. We call $\mu$ a \emph{measure of maximal $u$-entropy} if it satisfies
$$
h_\mu(f,\cF^{uu})=h(f,\cF^{uu}).
$$
By Hu, Wu, Zhu~\cite{HWZ}, the set $\MM^u(f)$ of measures of maximal $u$-entropy is always
non-empty, convex and compact. Moreover, its extreme points are ergodic measures.

\begin{main}\label{main}
If $f:M \to M$ is a $C^1$ diffeomorphism which factors over Anosov and has $c$-mostly contracting center,
then it admits finitely many ergodic measures of maximal $u$-entropy, and their supports are pairwise disjoint.
Each support has finitely many connected components, they are unions of
entire leaves of $\cF^{uu}$, and every leaf is dense in the corresponding
connected component.
\end{main}

The assumption that $f$ has $c$-mostly contracting center is $C^1$-open among diffeomorphisms that factor over
Anosov (see Proposition~\ref{p.robust}). Corresponding results were proved by~\cite{BoV00,BDP03,ANd10} for the
classical notion of mostly contracting center. It is not clear whether factoring over Anosov is also an open property.

Our construction provides an explicit description of the measures of maximal $u$-entropy in terms of their (leafwise)
densities relative to the reference measures. A different approach, using tools from geometric measure theory,
was recently developed by Climenhaga, Pesin, Zelerowicz~\cite{CPZ20} to prove existence and uniqueness for a
different class of partially hyperbolic dynamical systems.

We also point out that, while we state the theorem for (globally defined) diffeomorphisms,
the same arguments yield a version of the statement for $C^1$ embeddings.
Such a semi-global situation is indeed considered in Sections~\ref{s.example4} and~\ref{s.example5}.
The latter provides examples of measures of maximal $u$-entropy which are not measures of maximal entropy
(examples of measures of maximal entropy which do not maximize $u$-entropy are easy to exhibit~\cite{TaY19}).

In a forthcoming paper~\cite{UVYY2} we prove that the ergodic measures of maximal $u$-entropy constructed
in Theorem~\ref{main} satisfy a large deviations principle, and have exponential decay of correlations for
H\"older observables.

The present paper is organized as follows. In Section~\ref{s.definitions} we give precise definitions and
state two main results: Theorem~\ref{main.maximal.measures}, which characterizes the measures of maximal
$u$-entropy as the system's $c$-Gibbs $u$-states, and Theorem~\ref{main.skeleton}, which classifies the
ergodic $c$-Gibbs $u$-states. Their combination contains Theorem~\ref{main}, in a detailed form.
In Section~\ref{s.maximal.measures} we prove Theorem~\ref{main.maximal.measures}.
In Section~\ref{s.Gibbsu} we study the space of $c$-Gibbs $u$-states, and use the conclusions to
prove Theorem~\ref{main.skeleton}. Section~\ref{s.basic_properties} contains several properties of
diffeomorphisms with $c$-mostly contracting center, some of which are used in the sequences that follow.
In Sections~\ref{s.example1} to~\ref{s.example5} we exhibit several examples of diffeomorphisms
satisfying the assumptions of our results.
In the appendix, a few classical results about $C^2$ diffeomorphisms are extended to the $C^1$ case,
under the extra assumption that there exists a dominated splitting.

\section{Definitions and statements}\label{s.definitions}

In this section we give the precise definitions of the notions we are
going to use, and we state more detailed results which contain Theorem~\ref{main}.

\subsection{Partial hyperbolicity}

A $C^1$ diffeomorphism $f:M \to M$ on a compact manifold $M$ is \emph{partially hyperbolic}
if there exists a $Df$-invariant splitting
$$
T M = E^{cs} \oplus E^{uu}
$$
of the tangent bundle such that $Df \mid_{E^{uu}}$ is
uniformly expanding and \emph{dominates} $Df \mid_{E^{cs}}$.
By this we mean that there exist a Riemmanian metric on $M$
and a continuous function $\omega(x) < 1$ such that
\begin{equation}\label{eq.omega}
\|Df(x)v^{uu}\| \ge \frac{1}{\omega(x)}
\text{ and }
\frac{\|Df(x)v^{cs}\|}{\|Df(x)v^{uu}\|}
\le \omega(x)
\end{equation}
for any unit vectors $v^{cs}\in E^{cs}_x$ and $v^{uu}\in E^{uu}_x$,
and any $x \in M$.

It is a classical fact (see \cite{HPS77}) that the \emph{unstable} subbundle $E^{uu}$ is uniquely integrable,
meaning that there exists a unique foliation $\cF^{uu}$ which is invariant under $f$ and tangent to $E^{uu}$
at every point. We assume that $f$ is \emph{dynamically coherent}, meaning that the subbundle $E^{cs}$
is also uniquely integrable, and we denote by $\cF^{cs}$ the integral foliation.
A compact $f$-invariant set $\Lambda\subset M$ is \emph{$u$-saturated} if it consists of entire leaves of $\cF^{uu}$.
Then it is called \emph{$u$-minimal} if every strong-unstable
leaf contained in $\Lambda$ is dense in $\Lambda$.

In some of our examples, $E^{cs}$ itself splits into two continuous invariant subbundles, $E^{ss}$ and $E^{c}$, such that $Df \mid_{E^{ss}}$
is uniformly contracting and is \emph{dominated} by $Df \mid_{E^{c}}$:
$$
\|Df(x)v^{ss}\| \le \omega(x)
\text{ and }
\frac{\|Df(x)v^{ss}\|}{\|Df(x)v^{c}\|}
\le \omega(x)
$$
for any unit vectors $v^{ss}\in E^{ss}_x$ and $v^{c}\in E^{c}_x$,
and any $x \in M$.

\subsection{Markov partitions}\label{s.Markov.partitions}

Let $\cR=\{\cR_1, \dots, \cR_k\}$ be a Markov partition for the linear automorphism
$A:\TT^d\to\TT^d$. By this we mean (see Bowen~\cite[Section~3.C]{Bow75a})
a finite covering of $\TT^d$ by small closed subsets $\cR_i$ such that
\begin{enumerate}
   \item[(a)] each $\cR_i$ is the closure of its interior, and the interiors are pairwise disjoint;
   \item[(b)] for any $a, b \in\cR_i$, $W^u_i(a)$ intersects $W_i^s(b)$ at exactly one point,
which we denote as $[a,b]$;
   \item[(c)] $A(W^s_i(a)) \subset W^s_j(A(a))$ and $A(W^u_i(a)) \supset W^u_j(A(a))$ if $a$ is in
the interior of $\cR_i$ and $A(a)$ is in the interior of $\cR_j$.
\end{enumerate}
Here, $W^u_i(a)$ is the connected component of $W^u(a) \cap \cR_i$ that contains $a$,
and $W^s_i(a)$ is the connected component of $W^s(a) \cap \cR_i$ that contains $a$. We call them, respectively, the \emph{unstable plaque} and
the \emph{stable plaque} through $a$.
Property (b) is called \emph{local product structure}.

The boundary $\partial\cR_i$ of each $\cR_i$ coincides with
$\partial^s\cR_i \cup \partial^u\cR_i$, where $\partial^s\cR_i$ is the set of points
$x$ which are not in the interior of $W^u_i(x)$ inside the corresponding unstable leaf,
and $\partial^u\cR_i$ is defined analogously.
By the local product structure, $\partial^s\cR_i$ consists of stable plaques and
$\partial^u\cR_i$ consists of unstable plaques.
The Markov property (c) implies that the total stable boundary
$\partial^s\cR = \cup_i \partial^s\cR_i$ is forward invariant and the
total unstable boundary
$\partial^u\cR = \cup_i \partial^u\cR_i$ is backward invariant under $A$.
Since the Lebesgue measure on $\TT^d$ is invariant and ergodic for $A$,
it follows that both $\partial^s\cR$ and $\partial^u\cR$ have zero
Lebesgue measure. Then, by Fubini, the intersection of $\partial^s\cR$
with almost every unstable plaque has zero Lebesgue measure in the plaque.
It follows that the same is true for \emph{every} unstable plaque,
since the stable holonomies of $A$, being affine, preserve the class of
sets with zero Lebesgue measure inside unstable leaves.
A similar statement holds for $\partial^u\cR$.

Next, define $\cM = \{\cM_1,\ldots,\cM_k\}$ by $\cM_i = \pi^{-1}(\cR_i)$.
For each $i=1, \dots, k$ and $x\in\cM_i$, let $\xi_i^u(x)$ be the connected component of
$\cF^{uu}(x)\cap\cM_i$ that contains $x$, and $\xi_i^{cs}(x)$ be the pre-image of $W^{s}_i(\pi(x))$. By construction,
\begin{equation}\label{eq.Markov1}
f(\xi_i^{u}(x)) \supset \xi_j^{u}(f(x))
\text{ and }
f(\xi_i^{cs}(x)) \subset \xi_j^{cs}(f(x))
\end{equation}
whenever $x$ is in the interior of $\cM_i$ and $f(x)$ is in the interior of $\cM_j$.
We refer to $\xi_i^u(x)$ and $\xi_i^{cs}(x)$, respectively,
as the \emph{strong-unstable plaque} and the \emph{center-stable plaque}
through $x$.

The local product structure property also extends to $\cM$:
for any $x, y \in \cM_i$ we have that $\xi^u_i(x)$ intersects $\xi^{cs}_i(y)$ at exactly
one point, which we still denote as $[x,y]$. That can be seen as follows.
To begin with, we claim that $\pi$ maps $\xi_i^u(x)$ homeomorphically to $W^u_i(\pi(x))$.
In view of the assumption (b) above, to prove this it is enough to check that
$\pi(\xi_i^u(x)) = W^u_i(\pi(x))$. The inclusion $\subset$ is clear, as both sets are connected.
Since $\xi_i^u(x)$ is compact, it is also clear that $\pi(\xi_i^u(x))$ is closed in $W^u_i(\pi(x))$. To conclude, it suffices to check that it is also open in $W^u_i(\pi(x))$.
Let $b=\pi(z)$ for some $z\in\xi_i^u(x)$. By assumption (b), for any small neighborhood $V$ of $b$ inside $W^u(b)$, there exists a small neighborhood $U$ of $z$ inside $\cF^{uu}_z$
that is mapped homeomorphically to $V$. By definition, a point $w\in U$ is in $\cM_i$ if and
only if $\pi(w)$ is in $\cR_i$. Thus $\pi$ maps $U\cap\cM_i$ homeomorphically to $V\cap\cR_i$.
That implies that $b$ is in the interior of $W^u_i(b)$, and that proves that $\pi(\xi_i^u(x))$
is indeed open in $W^u_i(\pi(x))$. Thus the claim is proved.
Finally, $[x,y]$ is precisely the sole pre-image of $[\pi(x),\pi(y)]$ in $\xi^u(x)$;
notice that this pre-image does belong to $\xi_i^{cs}(y)$, by definition.

This shows that $\cM$ is a Markov partition for $f$, though not necessarily a generating one.
In any event, the fact that $f$ is uniformly expanding along strong-unstable leaves ensures that
$\cM$ is automatically \emph{$u$-generating}, in the sense that
\begin{equation}\label{eq_u-generating}
\bigcap_{n=0}^\infty f^{-n}\left(\xi_i^u\left(f^n(x)\right)\right)
= \{x\} \text{ for every } x\in\Lambda.
\end{equation}
We call \emph{center-stable holonomy} the family of maps $H^{cs}_{x,y}:\xi_i^u(x) \to \xi^u_i(y)$
defined by the condition that
$$
\xi_i^{cs}(z) = \xi^{cs}_i(H^{cs}_{x,y}(z))
$$
whenever $x, y \in \cM_i$ and $z \in \xi_i^u(x)$.

\subsection{Reference measures}

We call \emph{reference measures} the probability measures $\nu^u_{i,x}$
defined on each strong-unstable plaque $\xi_i^u(x)$, $x\in\cM_i$, $i\in\{1, \dots, k\}$ by
$$
\pi_*\nu^u_{i,x} = \Leb^u_{i,\pi(x)} = \text{normalized Lebesgue measure on } W^u_i(\pi(x)).
$$
Since the Lebesgue measure on unstable leaves are preserved by the stable holonomy of $A$ (as the
latter is affine), the construction in the previous section also gives that these reference measures
are preserved by center-stable holonomies of $f$:
\begin{equation}\label{eq_cs-invariant}
\nu_{i,y}^u = \left(H^{cs}_{x,y}\right)_*\nu^u_{i,x}.
\end{equation}
for every $x$ and $y$ in the same $\cM_i$. Similarly, the fact that Lebesgue measure on unstable
leaves has constant Jacobian for $A$ implies that the same is true for the reference measures of $f$:
if $f(\cM_i)$ intersects the interior of $\cM_j$ then
\begin{equation}\label{eq_constant_Jacobian}
f_*\left(\nu^u_{i,x} \mid_{f^{-1}\left(\xi_j^u(f(x))\right)}\right)
= \nu^u_{i,x}\left(f^{-1}\left(\xi_j^u(f(x))\right)\right) \nu^u_{j,f(x)}
\end{equation}
for every $x\in\cM_i\cap f^{-1}(\cM_j)$.
Note that $\nu^u_{i,x}\left(f^{-1}\left(\xi_j^u(f(x))\right)\right)$ is just the normalizing factor
for the restriction of $\nu^u_{i,x}$ to $f^{-1}\left(\xi_j^u(f(x))\right)$.

\begin{remark}\label{r.lowboundary}
Properties \eqref{eq_cs-invariant} and \eqref{eq_constant_Jacobian} imply that
$x \mapsto \nu^u_{i,x}(f^{-1}\xi_j^u(f(x)))$ is constant on $\cM_i \cap f^{-1}(\cM_j)$,
for any $i$ and $j$ such that $f(\cM_i)$ intersects the interior of $\cM_j$.
Thus this function takes only finitely many values.
\end{remark}

\begin{remark}\label{r.equivalentmeasures}
Let $x$ be on the boundary of two different Markov sets $\cM_i$ and $\cM_j$.
Then the restrictions of $\nu^u_{i,x}$ and $\nu^u_{j,x}$ to the intersection
$\xi^u_i(x) \cap \xi^u_j(x)$ are equivalent measures,
as they are both mapped by $\pi_*$ to multiples of the Lebesgue measure on
$W^u_i(\pi(x)) \cap W^u_j(\pi(x))$.
\end{remark}

\begin{remark}\label{r.zeroboundary}
As observed before, the intersection of $\partial^s\cR$ with every
unstable plaque $W^u_i(x)$ has zero Lebesgue measure inside the plaque.
Since $\pi$ sends each $\cM_i$ to $\cR_i$, with each $\xi_i^u(x)$ mapped homeomorphically to $W^u_i(\pi(x))$,
it follows that $\partial^s\cM \cap \xi^u_i(x)$ has zero $\nu^u_{i,x}$-measure for every $x\in\cM_i$ and every $i$.
\end{remark}

\subsection{Gibbs $u$-states}

An $f$-invariant probability measure $\mu$ is called a $c$-Gibbs $u$-state if its conditional
probabilities along strong-unstable leaves coincide with this family of reference measures $\nu_{i,x}^u$.
More precisely, for each $i$, let $\{\mu^u_{i,x}: x \in \cM_i\}$ denote the disintegration of
the restriction $\mu\mid_{\cM_i}$ relative to the partition $\{\xi_i^u(x): x\in \cM_i\}$.
Then we call $\mu$ a \emph{$c$-Gibbs $u$-state} if $\mu^u_{i,x}=\nu_{i,x}^u$ for $\mu$-almost every $x$.
The space of invariant $c$-Gibbs $u$-states of $f$ is denoted by $\Gibbs^u_c(f)$.
A few properties of this set are collected in Proposition~\ref{p.gibbs}.

\begin{main}\label{main.maximal.measures}
If $f:M \to M$ factors over Anosov then an $f$-invariant probability measure $\mu$
is a measure of maximal $u$-entropy if and only if it is a $c$-Gibbs $u$-state.
\end{main}

For the next result we need a condition that was already briefly mentioned in the Introduction:
$f$ has \emph{$c$-mostly contracting center} if for any $i\in\{1, \dots, k\}$ and $x\in\cM_i$
\begin{equation}\label{eq_c-mostly}
\limsup_n \frac 1n \log \|Df^n \mid_{E^{cs}_y}\| < 0
\end{equation}
for a positive $\nu^u_{i,x}$-measure subset of points $y\in\xi_i^u(x)$.
This condition is not used in Theorem~\ref{main.maximal.measures}.

\begin{remark}\label{r.terminology}
In our notation for Gibbs $u$-states and mostly contracting center, $c$- stands for \emph{constant},
referring to the fact that our reference measures have locally constant Jacobians.
The classical notions of \emph{Gibbs $u$-state}~\cite{PeS82,BoV00} and \emph{mostly contracting
center}~\cite{BoV00} were defined using instead the Lebesgue measures along strong-unstable leaves as
reference measures.
\end{remark}

Following~\cite{DVY16}, we call a \emph{skeleton} for $f$ any finite set $S = \{p_1, \ldots, p_m\}$
of periodic points such that:
\begin{enumerate}
\item[(a)] each $p_i$ is a hyperbolic saddle of $f$ with stable dimension equal to $\dim E^{cs}$;
\item[)b)] the strong-unstable leaf $\cF^{uu}(x)$ of any $x \in M$ has some transverse
intersection with the union of the stable manifolds through the orbits $\Orb(p_i)$;
\item[(c)] $W^u(p_i)\cap W^s(\Orb(p_j)) = \emptyset$ for any $i \ne j$.
\end{enumerate}
Recall that the \emph{homoclinic class} of a hyperbolic periodic point $p$ is the closure $H(p,f)$
of the set of all transverse intersections between the stable manifolds and the unstable manifolds of
the iterates of $p$.

Theorem~\ref{main} is contained in the following more detailed statement:

\begin{main}\label{main.skeleton}
Let $f:M \to M$ be as in Theorem~\ref{main}.
Then $f$ admits some skeleton $S$, and the number of ergodic $c$-Gibbs $u$-states is exactly $\# S$.

Each ergodic $c$-Gibbs $u$-state is supported on the closure of the strong-unstable leaves through
the orbit of some periodic point $p_i \in S$, and coincides with the homoclinic class of $p_i$.

Finally, these supports are pairwise disjoint, each one of them has finitely many connected
components, and each connected component is a $u$-minimal set.
\end{main}

\section{Proof of Theorem~\ref{main.maximal.measures}}\label{s.maximal.measures}

Let $f$ as in the assumptions of Theorem~\ref{main.maximal.measures}.
Keep in mind that, by Hu, Wu, Zhu~\cite{HWZ}, the set $\MM^u(f)$ of measures of maximal $u$-entropy
is non-empty, convex and compact, and its extreme points are ergodic measures.
Our goal in this section is to prove that $\MM^u(f)$ coincides with $\Gibbs^u_c(f)$.
We use $\Leb$ to denote the Lebesgue measure of $\TT^d$.

\begin{lemma}\label{l.topologicalentropy}
The topological $u$-entropy $h(f,\cF^{uu})$ of $f$ is equal to the topological entropy $h(A)$ of $A$.
\end{lemma}

\begin{proof}
Let $\Leb_f^u$ and $\Leb_A^u$ denote Lebesgue measure along, respectively, strong-unstable leaves of $f$
and unstable leaves of $A$. It is clear that we may find $C>1$ such that
\begin{equation}\label{eq_proporcional}
\frac{1}{C} \Leb^u_A(W^u_i(\pi(x))) \le \Leb^u_f(\xi_i^u(x)) \le C \Leb^u_A(W^u_i(\pi(x)))
\end{equation}
for every $i=1, \dots, k$ and every $x\in\cM_i$.
By construction, for every $n \ge 1$, the image $A^n(W_i^u(\pi(x)))$ is an (essentially disjoint)
union of unstable plaques $W_j^u(b)$ and, similarly, the image $f^n(\xi_i^u(x))$ is an
(essentially disjoint) union of strong-unstable plaques $\xi_j^u(y)$.
Moreover, there is a one-to-one correspondence between the unstable plaques $W_j^u(b)$ inside
$A^n(W_i^u(\pi(x)))$ and the strong-unstable plaques $\xi_j^u(y)$ inside $f^n(\xi_i^u(x))$ preserving the
index $j$. Thus, \eqref{eq_proporcional} implies that
$$
\frac{1}{C} \Leb^u_A\left(A^n(W^u_i(\pi(x)))\right) \le \Leb^u_f\left(f^n(\xi_i^u(x))\right) \le C \Leb^u_A\left(A^n(W^u_i(\pi(x)))\right)
$$
for every $n \ge 1$, $i=1, \dots, k$ and $x\in\cM_i$. This implies that the two sequences of volumes
grow at the same exponential rate, as claimed.
\end{proof}

Let $\cW^u$ be the unstable foliation of $A$.
Recall that $\cF^{uu}$ denotes the strong-unstable foliation of $f$.

\begin{remark}\label{r.affine}
The $u$-entropy is an affine function of the invariant measure:
\begin{equation}\label{eq.affine}
h_\mu(f,\cF^{uu}) = \int h_{\mu_P}(f,\cF^{uu}) \, d\hat\mu(P)
\end{equation}
where $\{\mu_P : P \in \cP\}$ is the ergodic decomposition of $\mu$ and $\hat\mu$ is
the associated quotient measure.
Indeed the corresponding statement for the classical notion of entropy is given by the Jacobs
theorem (\cite[Theorem~9.6.2]{FET16}), and the present version for partial entropy can be
deduced as follows (see~\cite[Chapter~5]{FET16} for context).
First, $\cP$ is the measurable partition characterized by the condition that points belong
to the same partition element if and only if they have the same Birkhoff average,
$\{\mu_P : P \in \cP\}$ is the Rokhlin disintegration of $\mu$ along $\cP$,
and $\hat\mu$ is the image of $\mu$ under the canonical map $M \to \cP$.
Let $\xi^u$ be any measurable partition subordinate to the strong-unstable foliation
relative to the measure $\mu$. Since Birkhoff averages are constant on
strong-unstable leaves, $\xi^u$ is finer than $\cP$.
By transitivity of the disintegration (see \cite[\S~1.7]{Rok67a} or
\cite[Exercise~5.2.1]{FET16}) it follows that the conditional probabilities
$\{\mu^u_{P,x}: x \in M\}$ of $\mu_P$ along the partition $\xi^u$ coincide with the
conditional probabilities $\{\mu^u_x: x \in \}$ of $\mu$ itself along $\cP$,
for $\hat\mu$-almost every $P$ and $\mu$-almost every $x$. It is clear that $\xi^u$
remains a subordinate partition relative to $\mu_P$ for $\hat\mu$-almost every $P$.
Now, the definition of the $u$-entropy means that
$$
\begin{aligned}
h_\mu(f,\cF^{uu})
& = \int_M \mu^u_x (f^{-1}\xi^u(f(x))) \, d\mu(x)\\
& = \int_\cP \int_M \mu^u_{P,x} (f^{-1}\xi^u(f(x))) \, d\mu_P(x) \, d\hat\mu(P)
 = \int_\cP  h_{\mu_P}(f,\cF^{uu}) \, d\hat\mu(P),
\end{aligned}
$$
as claimed.
\end{remark}

Denote $w=\pi_*\mu$. Then, since $(A,w)$ is a factor of $(f,\mu)$, it follows from Ledrappier, Walters~\cite{LeW77}
that $h_w(A) \le h_\mu(f)$. Let us point out that the inequality for the $u$-entropy that we
obtain in the next proposition points in the opposite direction.

\begin{proposition}\label{p.upperbelow}
Let $\mu$ be an invariant probability measure of $f$, and $w=\pi_*(\mu)$. Then $h_\mu(f,\cF^{uu})\leq h_w(A,\cW^u)$.
\end{proposition}

\begin{proof}
First, suppose that $h_w(A,\cW^u)=0$.
By~\cite[Proposition~2.5]{ViY17},  this happens if and only if there is a full $w$-measure subset $X$ that
intersects every unstable leaf of $A$ in at most one point. Then, $\pi^{-1}(X)$ intersects every strong-unstable
leaf of $f$ in at most one point, since $\pi$ is a homeomorphism from strong-unstable leaves of $f$ to
unstable leaves of $A$. Using~\cite[Proposition~2.5]{ViY17} once more, we conclude that $h_\mu(f,\cF^{uu})=0$.

From now on, suppose that $h_w(A,\cW^u)>0$. Initially, assume that $\mu$ is ergodic.
Let $\eta^u$ be any partition subordinate to
the unstable foliation of $A$. Then, by definition,
\begin{equation}\label{eq_igual1}
h_{w}(A,\cW^u)=H_{w}(A^{-1}\eta^{u} \mid \eta^{u}).
\end{equation}
Denote by $\xi^{cu}$ the pull-back $\pi^{-1}(\eta^u)$ of $\eta^u$.
Then, since $\pi_*(\mu)=w$, we have
\begin{equation}\label{eq_igual2}
H_{\mu}(f^{-1}\xi^{cu} \mid \xi^{cu}) = H_{w}(A^{-1}\eta^u \mid \eta^u).
\end{equation}
Next, consider $\xi^{uu} = \xi^{cu} \vee \cF^{uu}$, that is, the partition of $M$ whose elements
are the intersections of the elements of $\xi^{cu}$ with the strong-unstable leaves of $f$.
It is clear that $\xi^{uu}$ is finer than $\xi^{cu}$. Moreover,
our assumptions on $\pi$ ensure that $\xi^{uu}$ is subordinate to $\cF^{uu}$. Thus,
\begin{equation}\label{eq_igual3}
h_\mu(f,\cF^{uu})=H_{\mu}(f^{-1} \xi^{uu} \mid \xi^{uu}).
\end{equation}
We claim that
\begin{equation}\label{eq_igual4}
f^{-1} \xi^{uu} = f^{-1}\xi^{cu} \bigvee \xi^{uu}.
\end{equation}
Indeed, given any $x\in M$, both $f^{-1}\xi^{uu}(x)$ and $f^{-1}\xi^{cu}(x) \cap \xi^{uu}(x)$
are contained in the strong-unstable leaf through $x$.
Moreover, both are mapped to $A^{-1}\eta^u (\pi(x))$ under $\pi$.
Since $\pi$ is a homeomorphism on strong-unstable leaves, it follows that the two sets
coincide, as claimed.
Now, using Rokhlin~\cite[\S~5.10]{Rok67a}, it follows that
\begin{equation}\label{eq_igual5}
\begin{aligned}
H_{\mu}\left(f^{-1} \xi^{cu} \mid \xi^{cu}\right)
& \geq H_{\mu}\left(f^{-1} \xi^{cu} \mid \xi^{uu}\right)\\
& = H_{\mu}\left(f^{-1} \xi^{cu} \bigvee \xi^{uu} \mid \xi^{uu}\right)
= H_\mu\left(f^{-1} \xi^{uu} \mid \xi^{uu}\right).
\end{aligned}
\end{equation}
Combining this with \eqref{eq_igual1}, \eqref{eq_igual2} and \eqref{eq_igual3} we get
the statement of the proposition when $\mu$ is ergodic.

The general case now follows from the fact that the $u$-entropy is affine.
Indeed, \eqref{eq.affine} means that
\begin{equation} \label{eq.affine1}
h_\mu(f,\cF^{uu})= \int_\cP  h_{\mu_{P(x)}}(f,\cF^{uu}) \, d\mu(x).
\end{equation}
Analogously,
\begin{equation} \label{eq.affine2}
h_\mu(A,\cW^{u})= \int_\cQ  h_{w_{Q(a)}}(A,\cW^{u}) \, dw(a).
\end{equation}
Here $\cQ$ is the partition of $\TT^d$ defined by the condition that two points are
in the same partition element if and only if they have the same Birkhoff average for $A$.
It is clear that $\pi^{-1}(\cQ)$ is coarser than $\cP$.
Thus $\pi(P(x)))$ is contained in $\cQ(\pi(x))$ and $\pi_*\mu_{P(x)} = w_{Q(\pi(x))}$
for almost every $x$. By the ergodic case treated previously, we also have that
$$
h_{\mu_{P(x)}}(f,\cF^{uu}) \le h_{w_{Q(\pi(x))}}(A,\cW^{u}).
$$
Now the claim follows by combining this with \eqref{eq.affine1} and \eqref{eq.affine2}.
\end{proof}

Next, we want to prove that the equality in Proposition~\ref{p.upperbelow} carries substantial
rigidity. To state this in precise terms, let $\eta^u$ and $\xi^u$ be partitions as in
the proof of the proposition, and $\{w^u_a: a \in \TT^d\}$ and $\{\mu^u_x: x \in M\}$
be Rokhlin disintegrations of $w=\pi_*\mu$ and $\mu$, respectively, relative to those
partitions.

\begin{corollary}\label{c.upperbelow}
We have $h_\mu(f,\cF^{uu}) = h_w(A,\cW^u)$ if and only if $\pi$ sends conditional probabilities
of $\mu$ relative to $\xi^u$ to conditional probabilities  of $w$ relative to $\eta^u$, meaning that
$$
\pi_*\mu^u_x = w^u_{\pi(x)} \text{ for $\mu$-almost every $x$.}
$$
In particular, each $\mu^u_x$ is determined by $w^u_{\pi(x)}$.
\end{corollary}

\begin{proof}
Let $\{\mu^{cu}_y: y \in M\}$ denote a disintegration of $\mu$ relative to the partition $\xi^{cu}=\pi^{-1}(\eta^u)$.
Since $w=\pi_*\mu$, the fact that the
disintegration is essentially unique implies that
\begin{equation}\label{eq_igual6}
\pi_*\mu_y^{cu} = w^u_{\pi(y)} \text{ for $\mu$-almost every $y$.}
\end{equation}
Since $\xi^u$ is finer than $\xi^{cu}$, the conditional probabilities of each
$\mu^{cu}_y$ relative to the partition $\xi^u$ coincide $\mu$-almost everywhere with the
conditional probabilities $\mu^u_x$ of $\mu$ itself (transitivity of the disintegration,
see \cite[\S~1.7]{Rok67a} or \cite[Exercise~5.2.1]{FET16}).

First, we prove the 'only if' claim. Assume that  $h_\mu(f,\cF^{uu}) = h_w(A,\cW^u)$.
Then it follows from the relations \eqref{eq_igual1}--\eqref{eq_igual5} that
\begin{equation}\label{eq_igual7}
H_{\mu}(f^{-1} \xi^{cu} \mid \xi^{cu}) = H_{\mu}(f^{-1} \xi^{cu} \mid \xi^{u}).
\end{equation}
Since $\mu$ is $f$-invariant,
$$
H_{\mu}(f^{-n} \xi^{cu} \mid \xi^{cu})= nH_{\mu}(f^{-1} \xi^{cu} \mid \xi^{cu})
\text{ and }
H_{\mu}(f^{-n} \xi^{cu} \mid \xi^{u})= nH_{\mu}(f^{-1} \xi^{cu} \mid \xi^{u})
$$
for every $n \ge 1$. Thus, \eqref{eq_igual7} implies
\begin{equation}\label{eq_igual8}
H_{\mu}(f^{-n} \xi^{cu} \mid \xi^{cu}) = H_{\mu}(f^{-n} \xi^{cu} \mid \xi^{u})
\end{equation}
for every $n \ge 1$.
By Rokhlin~\cite[\S~5.10]{Rok67a}, this can only happen if $f^{-n} \xi^{cu}$ and $\xi^u$
are independent relative to $\xi^{cu}$, that is (see \cite[\S~1.7]{Rok67a}), if
$$
\mu^{cu}_y(A \cap B) = \mu^{cu}_y(A) \mu^{cu}_y(B)
$$
for any $f^{-n}\xi^{cu}$-measurable set $A$, any $\xi^u$-measurable set $B$,
and $\mu$-almost every $y$. See Figure~\ref{fig:Fig2}.

\begin{figure}[ht]
\psfrag{A}{$A$}\psfrag{B}{$B$}
\psfrag{x1}{$\xi^u(x)$}\psfrag{x2}{$\xi^{cs}(x)$}
\begin{center}
\includegraphics[height=1.8in]{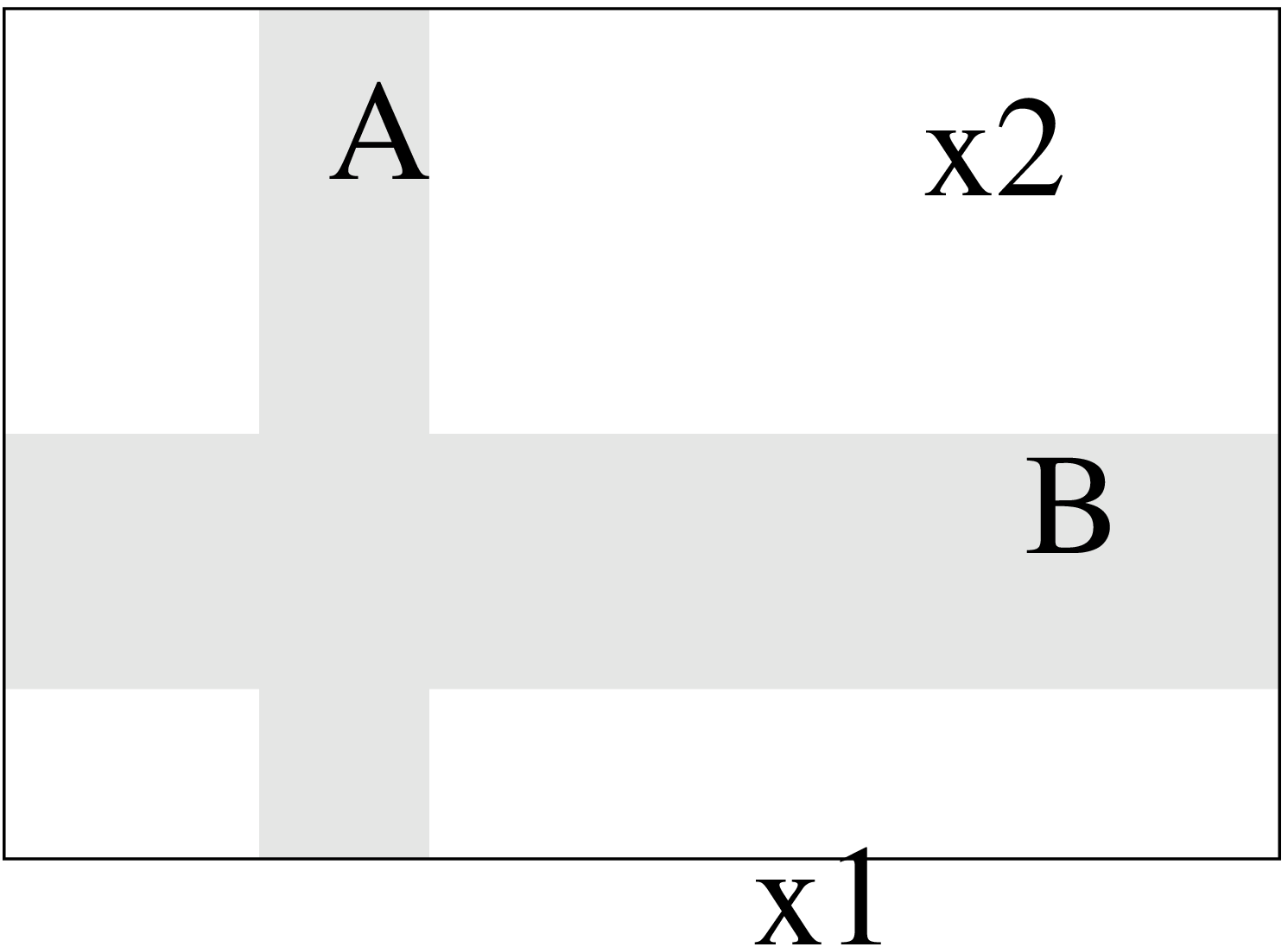}
\end{center}
\caption{\label{fig:Fig2}}
\end{figure}

By considering $B$ arbitrarily small, we see that the conditional measure $\mu_x^u(A)$
is independent of $x$, and so
\begin{equation}\label{eq_igual9}
\mu_x^u(A) = \mu^{cu}_x(A)
\end{equation}
for $\mu$-almost every $y$, $\mu^{cu}_y$-almost every $x$ in $\xi^{cu}(y)$,
and any $f^{-n}\xi^{cu}$-measurable set $A$. Moreover, \eqref{eq_igual6} implies that
$\mu^{cu}_x(A)=w_{\pi(x)}^u(\pi(A))$ for every $f^{-n}\xi^{cu}$-measurable set $A$.
This implies that
\begin{equation}\label{eq.three_measures}
\pi_*\mu_x^u=\pi_*(\mu_x^{cu}) = w_{\pi(x)}^u
\end{equation}
restricted to the $\sigma$-algebra of $f^{-n}\xi^{cu}$-measurable subsets of each $\xi^u(x)$,
for any $n$.
Now, since $f$ is (uniformly) expanding along strong-unstable leaves, the family of partitions
$\{f^{-n}\xi^{cu}: n\ge 1\}\cap \xi^u(x)$ generates the $\sigma$-algebra of all measurable
subsets of $\xi^u(x)$ for every $x$. Thus, the previous identity yields the `only if' statement:
$\pi_*\mu_x^u=\pi_*(\mu_x^{cu}) = w_{\pi(x)}^u$ for $\mu$-almost every $x$.

Next, assume that $\pi_*\mu^u_x = w^u_{\pi(x)}$ for $\mu$-almost
every $x$. Using \eqref{eq_igual6}, it follows that $\pi_*\mu^u_x = \pi_*\mu^{cu}_x$ for
$\mu$-almost every $x$, which implies that
\begin{equation}\label{eq.20bis}
\mu_x^u(A) = \mu^{cu}_x(A)
\end{equation}
for $\mu$-almost every $x$ and any $f^{-n}\xi^{cu}$-measurable set $A$. Compare \eqref{eq_igual9}.
Recalling that
$$
\begin{aligned}
H_{\mu}(f^{-1} \xi^{cu} \mid \xi^{cu})
& = \int H_{\mu^{cu}_x}\left(f^{-1}(\xi^{cu})\mid_{\xi^{cu}(x)}\right) \, d\mu(x) \text{ and}\\
H_{\mu}(f^{-1} \xi^{cu} \mid \xi^{u})
& = \int H_{\mu^{u}_x}(f^{-1}\left(\xi^{u})\mid_{\xi^{cu}(x)}\right) \, d\mu(x)
\end{aligned}
$$
we see that \eqref{eq.20bis} implies \eqref{eq_igual7}.
Using once more the relations \eqref{eq_igual1}--\eqref{eq_igual5},
we see that the later implies that $h_\mu(f,\cF^{uu}) = h_w(A,\cW^u)$.
This proves the `if' statement.

Finally, recall that $\pi$ is a homeomorphism onto its image restricted to each
strong-unstable leaf and, thus, to each element of $\xi^u$.
Thus, the relation $\pi_*\mu^u_x = w^u_{\pi(x)}$ means that each $\mu^u_x$ is the push-forward
of $w^u_{\pi(x)}$ under the inverse map $\pi^{-1}$.
This implies the last claim in the statement.
\end{proof}

\begin{corollary}\label{c.projection_mu}
If $\mu\in \MM^u(f)$ then $w=\pi_*\mu$ is the Lebesgue measure $\Leb$ of $\TT^d$.
In particular, $\mu$ is a $c$-Gibbs $u$-state.
\end{corollary}

\begin{proof}
By Proposition~\ref{p.upperbelow} and Lemma~\ref{l.topologicalentropy},
every $\mu\in \MM^u(f)$ satisfies
$$
h_w(A,\cW^u)\geq h_\mu(f,\cF^{uu})=h(f,\cF^{uu})=h(A).
$$
By \eqref{eq.two_entropies}, we also have that $h_w(A) \ge h_w(A,\cW^u)$.
It follows that $w$ is a measure of maximal entropy for $A$ which, by uniqueness,
implies that $w$ is the Lebesgue measure of $\TT^d$.
Since the total boundary $\partial\cR$ of the Markov partition $\cR$ has zero Lebesgue
measure, as observed at the beginning of Section~\ref{s.Markov.partitions},
it follows that the $\mu$-measure of $\partial\cM$ is equal to zero.
Let $\xi^u$ be the family of all $\xi^u_i(x)$ with $x$ in the interior of some $\cM_i$,
and, similarly, $\eta^u$ be the family of all $W^u_i(a)$ with $a$ in the interior of some
$\cR_i$. By the previous observation, these are partitions of (full measure subsets of)
$M$ and $\TT^d$, respectively. Moreover, $\eta^u$ is subordinate to the unstable foliation
of $A$, and $\xi^u$ is its pull-back under the map $\pi$ to each strong-unstable leaf
of $f$, which means that these two partitions verify the assumptions of
Proposition~\ref{p.upperbelow} and Corollary~\ref{c.upperbelow}.
This means that we are in a  position to apply the corollary, and conclude that
$$
\pi_*\mu_x^u = \Leb^u_{\pi(x)} = \text{ normalized Lebesgue measure on $\eta^u_{\pi(x)}$}
$$
for $\mu$-almost every $x$. This proves that $\mu$ is a $c$-Gibbs $u$-state.
\end{proof}

All that is left to finish the proof of Theorem~\ref{main.maximal.measures}, is to get the
converse statement, that is, that every $c$-Gibbs $u$-state is a measure of maximal $u$-entropy.
We will do that with the aid of the following lemma, which is of independent interest:

\begin{lemma}\label{l.Gibbs_decomposition}
Almost every ergodic component of a $c$-Gibbs $u$-state is also a $c$-Gibbs $u$-state.
\end{lemma}

\begin{proof}
Let $\{\mu_P: P \in\cP\}$ be the ergodic decomposition of a $c$-Gibbs $u$-state $\mu$,
and $\hat\mu$ be the corresponding quotient measure in $\cP$. This means that every $\mu_P$
is an ergodic mesure and
\begin{equation}\label{eq.disintegra1}
\mu = \int_\cP \mu_P \, d\hat\mu(P).
\end{equation}
For each $i$, let $\mu_i = \mu \mid_{\cM_i}$ and $\cP_i$ be the restriction of $\cP$ to
the domain $\cM_i$. For each $P\in\cP_i$, denote by $\mu_{P,i}$ the restriction of $\mu_P$
to $\cM_i$. Restricting \eqref{eq.disintegra1} to the measurable subsets of $\cM_i$ we see
that
\begin{equation}\label{eq.desintegra2}
\mu_i = \int_{\cP_i} \mu_{P,i} \, d\hat\mu(P).
\end{equation}

As before, we denote by $\{\mu^u_{i,x}: x \in \cM_i\}$ the disintegration of $\mu_i$ with
respect to the partition $\xi_i^u$ of $\cM_i$ into strong-unstable plaques.
The fact that Birkhoff averages are constant on strong-unstable leaves ensures that $\cP_i$
is a coarser partition than $\xi^u_i$ (see Remark~\ref{r.affine}).
Moreover, for each $P\in\cP_i$, let $\{\mu^u_{P,i,x}: x\in P\}$ be the disintegration of
$\mu_{P,i}$ with respect to the partition $\xi_i^u$.
By the transitivity of the disintegration (see Remark~\ref{r.affine}) and the definition
of $c$-Gibbs $u$-state
$$
\mu^u_{P(x),i,x} = \mu^u_{i,x} = \nu^u_{i,x} \text{ for $\mu_i$-almost every $x\in\cM_i$.}
$$
By \eqref{eq.desintegra2}, this implies that
$$
\mu^u_{P,i,x} = \nu^u_{i,x} \text{ for $\mu_{P,i}$-almost every $x$ and $\hat\mu$-almost every $P\in\cP_i$.}
$$
Since $i$ is arbitrary, this just means that $\mu_P$ is a $c$-Gibbs $u$-state for $\hat\mu$-almost every $P\in\cP$.
\end{proof}

\begin{corollary}\label{c.measurezero}
If $\mu$ is a $c$-Gibbs $u$-state of $f$ then $\mu(\partial\cM)=0$ and
$\mu$ is a measure of maximal $u$-entropy.
\end{corollary}

\begin{proof}
We begin by proving that $\mu(\partial\cM)=0$.
In view of Lemma~\ref{l.Gibbs_decomposition},
for this purpose it is no restriction to suppose that $\mu$ is ergodic.
Then
\begin{equation}\label{eq.ergodicity}
\frac{1}{n}\sum_{j=0}^{n-1}\delta_{f^j(y)} \to \mu
\end{equation}
for $\mu$-almost every $y$.
Consider any $i \in\{1, \dots, k\}$ such that the Markov element $\cM_{i}$ has
positive  $\mu$-measure. By the definition of $c$-Gibbs $u$-state, it follows
that \eqref{eq.ergodicity} holds for $\nu^u_{i,x}$-almost $y$ and $\mu$-almost
every $x\in\cM_{i}$. In particular, using the dominated convergence theorem,
$$
\frac 1n \sum_{j=0}^{n-1} f_*^j\nu_{i,x} \to \mu
$$
for $\mu$-almost every $x\in\cM_{i}$. By the definition of $c$-Gibbs $u$-state,
the image of each $\nu_{i,x}$ under the projection $\pi$ is the normalized
Lebesgue measure on $W^u_i(\pi(x))$, which we denote as $\Leb_{i,\pi(x)}$.
Thus, the previous relation implies that
$$
\frac 1n \sum_{j=0}^{n-1} A_*^j \Leb_{i,\pi(x)} \to \pi_*\mu.
$$
Now, it is well-known and easy to check that the left hand side converges to
the Lebesgue measure $\Leb$ on the torus $\TT^d$. So, we just proved that $\pi_*\mu=\Leb$.
As observed previously, the total boundary $\partial \cR$ has zero Lebesgue measure.
It follows that $\mu(\partial\cM)=0$, as claimed.

Let $\xi^u$ be the family of all $\xi^u_i(x)$ with $x$ in the interior of some $\cM_i$,
and, similarly, $\eta^u$ be the family of all $W^u_i(a)$ with $a$ in the interior of some
$\cR_i$. By the previous observation, these are partitions of (full measure subsets of)
$M$ and $\TT^d$, respectively. Moreover, $\eta^u$ is subordinate to the unstable foliation
of $A$, and $\xi^u$ is its pull-back under the map $\pi$ to each strong-unstable leaf
of $f$. By the definition of $c$-Gibbs $u$-state,
$$
\pi_*\mu_x^u = \Leb^u_{\pi(x)} = \text{ normalized Lebesgue measure on $\eta^u_{\pi(x)}$}
$$
for $\mu$-almost every $x$. Applying Corollary~\ref{c.upperbelow}, we conclude that
$h_\mu(f,\cF^{uu})=h_{\Leb}(A,\cW^u)$. Since $A$ is linear, we also have
that $h_{\Leb}(A,\cW^u)$ coincides with the topological $u$-entropy $h(A,\cW^u)$.
Moreover, by Lemma~\ref{l.topologicalentropy}, the two topological entropies
$h(A,\cW^u)$ and $h(f,\cF^{uu})$ are the same. Combining these identities, we
find that $h_\mu(f,\cF^{uu})=h(f,\cF^{uu})$. This means that $\mu$ is indeed
a measure of maximal $u$-entropy.
\end{proof}

At this point the proof of Theorem~\ref{main.maximal.measures} is complete.
Let us conclude this section with a few remarks on the broader significance of
Proposition~\ref{p.upperbelow} and Corollary~\ref{c.upperbelow}.

\begin{remark}\label{rk.invarainceprinciple}
The rigidity phenomenon stated in Corollary~\ref{c.upperbelow} is akin to some results
of Avila, Viana~\cite{Extremal} and Tahzibi, Yang~\cite{TaY19} for partially hyperbolic
diffeomorphisms $f$ whose center foliations $\cF^c$ are fiber bundles over the
quotient space $M/\cF^c$, as we are going to explain.
On the one hand, it follows from the invariance principle of Avila, Viana~\cite{Extremal}
that if the center Lyapunov exponents are non-positive then the disintegration of any
$f$-invariant probability measure $\mu$ along center leaves is preserved by strong-unstable
holonomies (\emph{$u$-invariance}).
On the other hand, Tahzibi, Yang~\cite{TaY19} proved that if
$h_\mu(f,\cF^{uu}_f)=h_\nu(g,\cF^{uu}_g)$, where $g$ is the quotient transformation induced
by $f$ on $M/\cF^c$ and $\nu$ is the projection of $\mu$, then the disintegration of $\mu$
along strong-unstable leaves is invariant under center holonomies (\emph{$c$-invariance},
see also \cite{AVW2}).
The two conclusions are the same, as $u$-invariance is known to be equivalent to $c$-invariance
(see Tahzibi, Yang~\cite[Lemma~5.2]{TaY19} and Viana, Yang~\cite[Proposition~6.2]{ViY19}).
The rigidity property in Corollary~\ref{c.upperbelow} is a broader kind of $c$-invariance,
with the fibers of $\pi$ taking the role of the center leaves.
\end{remark}

\section{$c$-Gibbs $u$-states}\label{s.Gibbsu}

The main purpose of this section is to prove Theorem~\ref{main.skeleton}.
To that end, we start by investigating some properties of the space of $c$-Gibbs $u$-states.
Then, we prove an alternative characterization of diffeomorphisms with $c$-mostly contracting center.
The proof of the theorem is given in the last subsection.

\subsection{Space of $c$-Gibbs $u$-states}\label{s.gibbs}

The next proposition collects some basic properties of the space $\Gibbs^u_c(f)$ of $c$-Gibbs
$u$-states of $f$:

\begin{proposition}\label{p.gibbs}
$\Gibbs^u_c(f)$  is non-empty, convex, and compact. Moreover,
\begin{enumerate}
\item almost every ergodic component of any $\mu \in \Gibbs^u_\nu(f)$
is a $c$-Gibbs $u$-state;
\item the support of every $\mu \in \Gibbs^u_\nu(f)$ is $u$-saturated;
\item for every $x \in \cM_i$ and $l \in\{1, \dots, k\}$,
every accumulation point of the sequence
$$
\mu_n = \frac{1}{n}\sum_{j=0}^{n-1} f^j_* \nu^u_{l,x}
$$
is a $c$-Gibbs $u$-state.
\end{enumerate}
\end{proposition}

\begin{proof}
We know from Theorem~\ref{main.maximal.measures} that  $\Gibbs^u_c(f)$ coincides with the space
$\MM^u(f)$ of measures of maximal $u$-entropy. It has been proven by Hu, Wu, Zhu~\cite{HWZ}
that the latter is non-empty, convex, and compact. The first statement in the proposition
follows directly. Claim (1) is given by Lemma~\ref{l.Gibbs_decomposition} above.

To prove claim (2), recall that by definition $\pi_*\mu^u_{i,x} = \Leb^u_{i,\pi(x)}$ for
$\mu$-almost every $x$. Since the restriction of $\pi$ to each strong-unstable leaf is a
homeomorphism, and $\Leb^i_{i,x}$ is supported on the whole unstable plaque $W^u_i(\pi(x))$,
we conclude that $\mu^u_{i,x}$ is supported on the whole $\xi^u_i(x)$ for $\mu$-almost
every $x$. This implies that the support of $\mu$ restricted to each $\mu_i$ is a union
of complete strong-unstable plaques $\xi^u_i(x)$.
We know from Corollary~\ref{c.measurezero} that $\mu(\partial\cM)=0$.
Let $Y_\delta$ be the set of points $x$ such that, for some $i\in\{1, \dots, k\}$, the
plaque $\xi^u_i(x)$ is contained in $\supp\mu$, and its boundary is $\delta$-away from $x$.
It follows that $\mu(Y_\delta)\to 1$ as $\delta\to 0$.
By Poincar\'e recurrence, for $\mu$-almost every $x \in Y_\delta$ there are infinitely
many values of $n \ge 1$ such that $f^{-n}(x) \in Y_\delta$. Thus, iterating the
plaques $\xi^u(f^{-n}(x))$ forward, we conclude that the whole strong-unstable leave $x$
is contained in the support of $\mu$. Next, making $\delta\to 0$ we get that this conclusion
holds for $\mu$-almost every $x$. Thus the support of $\mu$ is $u$-saturated, as claimed.

Now we prove claim (3). Begin by noticing that
$$
\pi_*\mu_n = \frac 1n \sum_{j=0}^{n-1} A^j_*\Leb^u_{i,\pi(x)}
$$
converges to the Lebesgue measure of $\TT^d$, since $A$ is a linear Anosov diffeomorphism.
Thus, $\pi_*\mu$ is the Lebesgue measure. In particular, $\mu(\partial\cM)=0$.
Property \eqref{eq_constant_Jacobian} ensures that the restriction of each iterate
$f^j_*\nu^u_{l,x}$ to every $\cM_i$ is a linear combination of reference measures.
Then the same is true for each $\mu_n \mid_{\cM_i}$.
Property \eqref{eq_cs-invariant} means that all these reference measures inside each $\cM_i$
are the same up to $cs$-holonomy.
It follows that for any accumulation point $\hat\mu_i$ of the sequence $\mu_n \mid_{\cM_i}$,
its conditional probabilities are also given by the reference measure.
The fact that $\mu(\partial\cM_i)=0$ also implies that $\hat\mu_i$ coincides with the
restriction of $\mu$ to $\cM_i$. This proves the claim.
\end{proof}

\subsection{Maps with $\nu$-mostly contracting center}

In the remainder of this section, the diffeomorphism $f:M\to M$ is assumed to be as in
Theorem~\ref{main}. First, we prove:

\begin{proposition}\label{p.contracting}
$f$ has $c$-mostly contracting center if and only if all center-stable Lyapunov exponents
of every ergodic $c$-Gibbs $u$-state of $f$ are negative.
\end{proposition}

\begin{proof}
The 'only if' statement is easy. Indeed, suppose that $f$ has $c$-mostly contracting center,
and let $\mu$ be any ergodic $c$-Gibbs $u$-state.
The definition of $c$-mostly contracting center implies that the center-stable Lyapunov exponents
are negative on positive $\nu^u_x$-measure subsets of every strong-unstable manifold.
Since the conditional probabilities $\mu^u_{i,x}$ of $\mu$ are given by the reference measures,
it follows that the center-stable Lyapunov exponents are negative on a positive $\mu$-measure subset
and, thus, $\mu$-almost everywhere.

Now suppose that the center-stable Lyapunov exponents of every ergodic $c$-Gibbs $u$-state are
negative. Take any $i$ and any $x \in \cM_i$.
By part (3) of Proposition~\ref{p.gibbs}, there exists a sequence $(n_m)_m\to\infty$ such that
$$
\mu_{n_m} = \frac{1}{n_m}\sum_{0}^{n_m-1} f^l_*\nu^u_{i,x}
$$
converges to a $c$-Gibbs $u$-state $\mu$. Since the ergodic components of $\mu$ are also
$c$-Gibbs $u$-states, our assumptions imply that the center-stable Lyapunov exponents are negative at
$\mu$-almost point. Then it follows from~\cite[Theorem~3.11]{ABC05} that $\mu$-almost every
point has a Pesin local stable manifold with dimension equal to $\dim E^{cs}$.

We also know from Corollary~\ref{c.measurezero} that $\mu(\partial\cM)=0$.
So, keeping in mind that the conditional probabilities of $\mu$ with respect to the partition
$\xi^u_j$ on each $\cM_j$ coincide with the reference measures,
we can find $j$, $z \in \supp\mu$ in the interior of $\cM_j$, $\epsilon>0$,
and a positive $\nu^u_{j,z}$-measure set $\Gamma\subset\xi^u_{j,z}$ such that every
$y \in \Gamma$ has a Pesin local stable manifold of size greater than $\epsilon$.
Observe that the holonomy of the Pesin stable lamination coincides with the center-stable holonomy,
and so it preserves the reference measures.

Property \eqref{eq_constant_Jacobian} ensures that the restriction of $\mu_{n_m}$ to $\cM_j$
is a linear combination of reference measures.
Since $z$ is in the support of $\mu=\lim_k \mu_{n_m}$ this linear combination is not void,
indeed there are iterates of $\xi_{i,x}^u$ close to $\xi^u_{j,z}$.
In particular, recalling property~\eqref{eq_cs-invariant} we see that for every $k$ sufficiently
large there exists a positive $\nu^u_{i,x}$-measure subset of points $y\in\xi^u_{i,x}$
such that $f^l(w)$ belongs to the Pesin local stable manifold of some point in $\Gamma$
for some $l < n_m$. This implies that
$$
\lim\frac{1}{n} \log \|Df^n \mid _{E^{cs}_w}\|<0.
$$
Since $i$ and $x$ are arbitrary, this proves that $f$ has $c$-mostly contracting center.
\end{proof}

\begin{corollary}\label{c.basin}
If $f$ has $c$-mostly contracting center then every $\xi^u_{i,x}$, $i\in\{1, \dots, k\}$
and $x \in \cM_i$,
has a positive $\nu^u_{i,x}$-measure subset of points contained in the basin of some ergodic
$c$-Gibbs $u$-state.
\end{corollary}

\begin{proof}
In the second part of the proof of Proposition~\ref{p.contracting}, consider an ergodic component $\tilde\mu$ of $\mu$ which is a $c$-Gibbs $u$-state (recall part (1) of Proposition~\ref{p.gibbs}),
we can take $\Gamma$ to  consist of points $y$ in the \emph{basin} of $\tilde\mu$, that is,
such that
$$
\frac 1n \sum_{l=0}^{n_m-1} \delta_{f^l(y)} \to \tilde\mu.
$$
Then, clearly, the Pesin local stable manifold $W^s(y)$ of every $y\in\Gamma$ is contained in
the basin of $\tilde\mu$. Thus, the argument in the proof of the proposition also proves that
for every $k$ sufficiently large there exists a positive $\nu^u_{i,x}$-measure subset of points
$y\in\xi^u_{i,x}$ which are in the basin of $\tilde\mu$.
\end{proof}

\subsection{Proof of Theorem~\ref{main.skeleton}}

This will consist of several auxiliary lemmas:

\begin{lemma}\label{l.finitegibbs}
There are finitely many ergodic $c$-Gibbs $u$-states for $f$.
\end{lemma}

\begin{proof}
Let $\mu_n$, $n\ge 1$ be any sequence of ergodic $c$-Gibbs $u$-states.
By compactness (Proposition~\ref{p.gibbs}), it is no  restriction to assume that $(\mu_n)_n$
converges to some $c$-Gibbs $u$-state $\mu$.
By Corollary~\ref{c.measurezero}, the boundary set $\partial\cM$ has zero $\mu$-measure.
Then we can find $x\in \sup\mu$ in the interior of some $\cM_i$.
By Corollary~\ref{c.basin}, there is a positive $\nu^u_{i,x}$-measure set $\Gamma_{i,x}$
formed by points which have Pesin local stable manifolds and belong to the basin of some
ergodic $c$-Gibbs $u$-state $\tilde\mu$.
Then all nearby plaques $\xi^u_{i,y}$ intersect the local stable manifolds through
$\Gamma_{i,x}$ on positive $\nu^u_{i,y}$-measure subsets. That is because the projections
along the local stable manifolds correspond to the center-stable holonomies of $f$,
as we observed previously, and the center-stable holonomies preserve the references measures,
according to \eqref{eq_cs-invariant}.
Since the $\mu_n$ converge to $\mu$, their supports must accumulate on $\xi^u_{i,x}$.
Keeping in mind that these are all $c$-Gibbs $u$-states, that is, that their conditional
measures relative to the partition $\xi^u_i$ are given by the reference measures,
it follows that these intersections have positive $\mu_n$-measure for every large $n$.
By ergodicity, this implies that $\mu_n=\tilde\mu$ for every large $n$.
Thus the family of ergodic $c$-Gibbs $u$-states is finite after all.
\end{proof}

\begin{lemma}\label{l.disjoint}
The supports of distinct ergodic $c$-Gibbs $u$-states are pairwise disjoint.
\end{lemma}

\begin{proof}
Let $\mu_1$ and $\mu_2$ be ergodic $c$-Gibbs $u$-states, and suppose that their supports
intersect at some point $x$. Consider any $i\in\{1, \dots, k\}$ such that $x\in\cM_i$.
By Corollary~\ref{c.measurezero}, there exists a positive $\nu^u_{i,x}$-measure subset
$\Gamma_{i,x}$ of $\xi^u_{i,x}$ consisting of points with Pesin local stable manifolds
inside the basin of some ergodic $c$-Gibbs $u$-state $\tilde\mu_i$.

Since there are finitely many Markov elements, we can fix $i$ such that the support of
$\mu_1$ accumulates on $x$ inside $\cM_i$, meaning that, every relative neighborhood of
$x$ in $\cM_i$ has positive $\mu_1$-measure. Since the support is $u$-saturated, and
the conditional probabilities are given by the reference  measures $\nu^u_{i,y}$,
which are preserved by the center-stable holonomies, it follows that there exists a
positive $\mu_1$-measure set $\Gamma_1$ consisting of points in Pesin stable manifolds
through points of $\Gamma_{i,x}$. Then $\Gamma_1$ is contained in the basin of
$\tilde\mu_i$ which, by the ergodicity of $\mu_1$, implies that $\mu_1=\tilde\mu_i$.

If $x$ is in the interior of $\cM_i$ then the support of $\mu_2$ also accumulates on it
inside $\cM_i$. Then the same argument proves that $\mu_2=\tilde\mu_i$, which yields
the conclusion of the lemma. In general, we argue as follows.
Since $\Gamma_{i,x}$ has positive $\nu^u_{i,x}$-measure, the support
$$
\supp\left(\nu^u_{i,x} \mid_{\Gamma_{i,x}}\right)
$$
is a subset of $\xi^u_i(x)$ with positive $\nu^u_{i,x}$-measure.
Thus, using also Remark~\ref{r.zeroboundary}, we may find $y$ in this support but not in
$\partial^s\cM$. In other words, there exists a neighborhood $U_y$ of $y$ inside the
strong-unstable leaf such that $U_y \cap \Gamma_{i,x}$ has positive
$\nu^u_{i,x}$-measure in $\xi^u_i(x)$,
and $U_y$ is contained in $\xi^u_j(y)$ for any $j$ such that $\cM_j$ contains $y$.
Since the support of $\mu_2$ accumulates on $y$, and the Markov partition is finite,
there must exist $j$ such that $\mu_2$ accumulates on $y$ inside $\cM_j$.
By Remark~\ref{r.equivalentmeasures}, the set  $U_y \cap \Gamma_{i,x}$ has positive
$\nu^u_{j,y}$-measure in $\xi^u_j(y)$. Keep in mind that all its points have Pesin
local stable manifolds inside the basin of $\tilde\mu_i$.
Thus, since the support of $\mu_2$ accumulates on $\xi^u_j(y)$ inside $\cM_j$,
the same argument as before gives that $\mu_2=\tilde\mu_i$.
This proves that $\mu_1 = \mu_2$.
\end{proof}

\begin{corollary}\label{c.disjoint}
If $\mu$ and $\mu_0$ are $c$-Gibbs $u$-states such that $\mu_0$ is ergodic and
$\supp\mu \subset \supp\mu_0$ then $\mu=\mu_0$.
\end{corollary}

\begin{proof}
Using Proposition~\ref{p.gibbs}(1), the ergodic components of $\mu$ are (ergodic)
$c$-Gibbs $u$-states, supported inside $\supp\mu_0$.
By Lemma~\ref{l.disjoint}, this implies that all the ergodic components coincide
with $\mu_0$, and so $\mu=\mu_0$.
\end{proof}

An invariant probability measure $\mu$ of a diffeomorphism $f:M\to M$ is said to be
\emph{hyperbolic} if the Lyapunov exponents of $f$ are non-zero at $\mu$-almost every point.
We denote by $O^s$ and $O^u$ denote the sums of the the Oseledets subbundles (defined
$\mu$-almost everywhere) associated
to  the negative and positive Lyapunov exponents, respectively.
We will need the following $C^1$ version of the shadowing lemma of Katok~\cite{Kat80c}:

\begin{lemma}\label{l.katok_shadowing_lemma}
Let $f:M\to M$ be a $C^1$ diffeomorphism and $\mu$ be an ergodic hyperbolic probability measure.
Assume that the Oseledets splitting $O^s \oplus O^u$ extends to a dominated splitting of the
tangent bundle over the support of $\mu$.
Then $\mu$ is the weak$^*$ limit of a sequence of invariant measures $\mu_n$ supported on the
orbits of hyperbolic periodic points $p_n$ such that the sequence $(p_n)_n$ converges to $\supp\mu$,
each $p_n$ has $\dim O^s$ contracting eigenvalues, and their stable manifolds have uniform size.
\end{lemma}

This is part of a more detailed result which was proved in \cite{Yang-tese}.
Since that paper will not be published, we reproduce the proof in the appendix (see Section~\ref{a.katok}).

\begin{lemma}\label{l.periodic}
Let $\mu$ be an ergodic $c$-Gibbs $u$-state of $f$.
Then there is a hyperbolic periodic orbit $\Orb(p)$ contained in $\supp\mu$ and having exactly
$\dim E^{cs}$ contracting eigenvalues. Moreover, the union of the strong-unstable leaves through
the points of $\Orb(p)$ is dense in $\supp\mu$.
\end{lemma}

\begin{proof}
By Proposition~\ref{p.contracting}, the center-stable Lyapunov exponents of $\mu$ are all negative,
and that implies that $\mu$ is a hyperbolic measure.
Using Lemma~\ref{l.katok_shadowing_lemma}, we find a sequence of hyperbolic periodic points $p_n$
with $\dim E^{cs}$ contracting eigenvalues, and stable manifolds of uniform size,
converging to some point $x \in \supp\mu$. In particular, for every large $n$ the stable manifold
of $p_n$ intersects the strong-unstable leaf of $x$.
By Proposition~\ref{p.gibbs}(2), it follows that the stable manifold of $p_n$ intersects $\supp\mu$.
Since the support is invariant and closed, it follows that $p_n \in \supp\mu$ for every large $n$.
This gives the first part of the claim.

Proposition~\ref{p.gibbs}(2) also gives that the strong-unstable leaves through the points of
$\Orb(p)$ are contained in $\supp\mu$. To prove that their union $\cF^{uu}(\Orb(p))$ is dense
in the support,  fix $i$ such that $p\in\cM_i$ and then consider any accumulation point $\tilde\mu$
of the sequence
$$
\frac{1}{n} \sum_{j=0}^{n-1}f^j_*\nu^u_{i,p}.
$$
It is clear that the support of $\tilde\mu$ is contained in the closure of the $\cF^{uu}(\Orb(p))$.
Also, by Proposition~\ref{p.gibbs}, $\tilde\mu$ is a $c$-Gibbs $u$-state, and so is almost every
ergodic component of it. Let $\bar\mu$ be any ergodic component with $\supp\bar\mu$.
It is no restriction to assume that $\supp\bar\mu$ is contained in $\supp\tilde\mu$ which is itself contained in $\supp\mu$.
Since the supports of ergodic $c$-Gibbs $u$-states are disjoint
(Lemma~\ref{l.disjoint}), it follows that $\bar\mu = \mu$.
Then, as the ergodic component is arbitrary, we get that $\tilde\mu = \mu$,
which proves the second part of the claim.
\end{proof}

\begin{lemma}\label{l.minimal}
The support of every ergodic $c$-Gibbs $u$-state $\mu$ has finitely many connected components,
and each of them is $u$-minimal.
Moreover, for any hyperbolic periodic point $p$ in the support and having exactly $\dim E^{cs}$
contracting eigenvalues, $\supp\mu$ coincides with the closure of the unstable manifold of
$\Orb(p)$, which is also equal to the homoclinic class of $p$.
\end{lemma}

\begin{proof}
Let $p\in \supp\mu$ be as in the statement.
We already know, from Proposition~\ref{p.gibbs}(2), that $\supp\mu$ is
$u$-saturated.
Let $F = f^{\pi(p)}$ where $\pi(p)$ denotes the period of $p$.
It is clear that $\cM$ is also a Markov partition for $F$, and thus
the two maps $f$ and $F$ have the same reference measures.
Since the inequality \eqref{eq.c_mostly_contracting} is clearly
inherited by iterates, it follows that $F$ has $c$-mostly contracting
center. Since $\mu$ is ergodic for $f$, its ergodic decomposition for
the iterate $F$ has the form
\begin{equation}\label{eq.ergodic_decomposition}
\mu = \frac{1}{l}\left(\mu_0+\cdots + f^{l-1}_*\mu_0\right),
\end{equation}
where $l$ is some divisor of $\pi(p)$,
and $\mu_0$ is an ergodic invariant measure for $f^l$.
Applying Lemma~\ref{l.disjoint} to $F$, we get that the
supports $\supp(f^j_*\mu_0)$, $j=0, \ldots, l-1$ are pairwise disjoint.
It is no restriction to assume that $p\in\supp\mu_0$.
Let $x$ be any point in $\supp\mu_0$ and $i=1, \dots, k$ such that
$x\in\cM_i$. Arguing as in the second part of the proof of
Lemma~\ref{l.periodic}, with $x$ in the place of $p$, we get
that
$$
\frac{1}{n}\sum_{j=0}^{n-1}F^j_*\nu^u_{i,x}
\text{ converges to } \mu_0 \text{ as $n\to\infty$.}
$$
This proves that the orbit $\cup_{n \ge 0} F^n(\cF^{uu}(x))$ of
the strong-unstable leaf through $x$ is dense in $\supp\mu_0$.

In particular, noting that $p$ is fixed point for $F$, the leaf
$\cF^{uu}(p)$ is dense in $\supp\mu_0$. Hence, the support of
$\mu_0$ is connected. This implies that the connected components
of $\supp\mu$ are precisely the $\supp(f^j_*\mu_0)$,
$j=0, \ldots, l-1$. In particular, they are finitely many.
To finish the proof we must deduce that each of these supports
is $u$-minimal, and for that it suffices to check that every
strong-unstable leaf $\cF^{uu}(x)$ is dense in $\supp\mu_0$.

Since the orbit of $\cF^{uu}(x)$ is dense,
by the first paragraph of the proof, there exists $n_x\ge 1$ such that
$F^{n_x}(\cF^{uu}(x))$ intersects (transversely) the stable manifold
$W^s(p)$. Since the latter is a fixed set, we get that $\cF^{uu}(x)$
itself intersects $W^s(p)$ transversely.
Let $V^s_N(p)$, $N\ge 1$ be an increasing sequence of open, relatively
compact neighborhoods of $p$ inside its stable manifold whose union is
the whole $W^s(p)$.
For each $N\ge 1$, the set of points $x\in \supp\mu_0$ whose
strong-unstable leaf intersects $V^s_N(p)$ transversely is an open
subset of the support.
Thus, by compactness of $\supp\mu_0$, one may find $N \ge 1$ such
that $\cF^{uu}(x)$ intersects $V^s_N(p)$ for every $x\in\supp\mu_0$.
Applying this conclusion to the backward iterates $\cF^{uu}(F^{-n}(x))$,
$n \ge 1$, and using the inclination lemma
(see \cite[Lemma~II.7.1]{PdM82_EN}),
we conclude that the closure of $\cF^{uu}(x)$ contains the unstable
manifold $W^u(p)=\cF^{uu}(p)$.
Therefore, $\cF^{uu}(x)$ is dense in the support of $\mu_0$,
we wanted to prove.

It is clear that the homoclinic class of $p$ is contained in the closure of
the unstable manifold of $\Orb(p)$ which, by $u$-minimality coincides with
the support of $\mu$. Thus, to prove the last part of the lemma,
we only need to show that the unstable manifold of every iterate of $p$
is contained in the homoclinic class.
Let $x$ be a point in the unstable manifold $W^u(f^i(p))$ of any iterate $f^i(p)$,
and $U$ be an arbitrarily small neighborhood of $x$ inside $W^u(f^i(p))$.
By $u$-minimality and the fact that $F$ is expanding along unstable manifolds,
there exists $n \ge 1$ such that $F^n(U)$ meets the local stable manifold of $p$
transversely, which means that $F^n(U)$ contains a point in the homoclinic class.
It follows that $x$ is in the homoclinic class, since this is an invariant closed set.
\end{proof}


\begin{lemma}\label{l.skeleton}
Let $\mu_1,\ldots, \mu_m$ be the set of ergodic $c$-Gibbs $u$-states of $f$
and, for each $i$, let $p_i$ be a hyperbolic periodic point in $\supp\mu_i$
as in Lemma~\ref{l.periodic}.
Then $\{p_1, \ldots, p_m\}$ is a skeleton for $f$.
\end{lemma}

\begin{proof}
By Lemma~\ref{l.periodic}, the union of the strong-unstable leaves
$\cF^{uu}(f^n(p_i)) = W^s(f^n(p_i))$ through the points of $\Orb(p_i)$
is dense in $\supp\mu_i$, for each $i$.
By Lemma~\ref{l.disjoint}, the supports $\supp\mu_i$, $1\le i \le m$ are
pairwise disjoint. Thus, in view of the inclination lemma,
$$
W^s(\Orb(p_i)) \cap W^u(\Orb(p_j)) = \emptyset
$$
whenever $i \ne j$. This gives condition (c) in the definition of a
skeleton. Condition (a) is clear from the choice of the $p_i$.

Thus, to finish proving that $\{p_1, \ldots, p_m\}$ is a skeleton we
only have to prove condition (b), that is, that that every strong-unstable
leaf $\cF^{uu}(x)$ has a transverse with $W^s(\Orb(p_i))$ for some $i$.
Given $x$, fix $l=1, \dots, k$ such that $x\in\cM_l$, and let $\mu$ be any
accumulation point of the sequence
$$
\frac{1}{n}\sum_{i=0}^{n-1} f^i_*\nu^u_{l,x}.
$$
By Proposition~\ref{p.gibbs}, this is a $c$-Gibbs $u$-state, and so are
its ergodic components. Thus $\mu$ may be written as
$$
\mu = \sum_{s=1}^{m} a_s \mu_s.
$$
It is no restriction to assume that $a_1 \ne 0$.
Using the fact that $p_1 \in \supp\mu_1$, we conclude that there for any
$r>0$ there exists $n$ arbitrarily large large such that
$$
\left(f^n_*\nu^u_{l,x}\right)\left(B_r(p_1)\right)>0.
$$
This implies that $f^n(\cF^{uu}_{loc}(x))$ intersects $B_r(p_1)$, for
any $r>0$. taking $r>0$ sufficiently small, this guarantees that
$f^n(\cF^{uu}_{loc}(x))$ has a transverse intersection with $W^s(\Orb(p_1))$.
Since $W^s(\Orb(p_1))$ is an invariant set, it follows that
$\cF^{uu}_{loc}(x)$ itself has a transverse intersection with $W^s(\Orb(p_1))$.
\end{proof}

Theorem~\ref{main.skeleton} is contained in Lemmas~\ref{l.finitegibbs} to~\ref{l.skeleton}.

\section{Openess of $c$-mostly contracting center}\label{s.basic_properties}

Next we prove that the property of having $c$-mostly contracting center is open among
partially hyperbolic differmorphisms which factor to a given Anosov automorphism.
This fact will be needed in Section~\ref{s.example5}.

\begin{proposition}\label{p.robust}
Let $f:M\to M$ be a $C^1$ partially hyperbolic diffeomorphism on a compact manifold $M$.
Suppose there is a $C^1$-neighborhood $\cU$ of $f$ such that all diffeomorphisms in $\cU$
factor over the same Anosov automorphism $A:\TT^d\to\TT^d$.
If $f$ has $c$-mostly contracting center then any diffeomorphism $g$ which
is sufficiently $C^1$-close to $f$ also has $c$-mostly contracting center.
\end{proposition}

\begin{proof}
Let $(f_n)_n$ be a sequence of diffeomorphisms converging to $f$ in the $C^1$ topology.
We want to prove that $f_n$ has $c$-mostly contracting center for every large $n$.
For this, by Proposition~\ref{p.contracting}, we only need to show that the center Lyapunov
exponents for any ergodic $c$-Gibbs $u$-state of $f_n$ are all negative, if $n$ is large enough.

Consider any $g\in\cU$.
By Lemma~\ref{l.topologicalentropy}, its topological u-entropy $h(g, \cF^{uu})$
is equal to the topological entropy $h(A)$ of $A$.
Recall that (Hu, Wu, Zhu~\cite{HWZ}) the set $\MM^u(g)$ of measures of maximal $u$-entropy
is non-empty, convex and compact, and its extreme points are ergodic measures.
By Theorem~\ref{main.maximal.measures}, $\MM^u(g)$ coincides with the space of $c$-Gibbs
$u$-states, for any $g\in \cU$. Thus, denoting by $\cF^{uu}_g$ the strong-unstable
foliation of $g$,
$$
h_\mu(g,\cF^{uu}_g) \leq h(A)
$$
for any invariant probability $\mu$ of $g$, and the identity holds if and only if
$\mu$ is a $c$-Gibbs $u$-state.

Let $(\mu_n)_n$ be a sequence of probability measures such that each $\mu_n$ is $f_n$-invariant
and the weak$^*$ limit $\mu=\lim_n \mu_n$ exists.
Then $\mu$ is $f$-invariant and, by Yang~\cite{Yan16},
$$
\limsup_n h_{\mu_n}(f_n,\cF^{uu}_{f_n}) \leq h_\mu(f,\cF^{uu}_f).
$$
Together with the previous paragraph, this implies that
\begin{equation}\label{eq.convergencyGibbs}
\limsup_n \Gibbs^u_c(f_n) \subset \Gibbs^u_c(f).
\end{equation}

\begin{lemma}\label{l.newcriterion}
A diffeomorphism $g\in \cU$ has $c$-mostly contracting center if and only if there are
$m \ge 1$ and $a <0$ such that
$$
\frac 1m \int_M \log \|Dg^m\mid_{E_g^{cs}}\| \, d\mu < a
\text{ for any ergodic $c$-Gibbs $u$-state $\mu$ of $g$.}
$$
\end{lemma}

Let us assume this fact for a while, and use it deduce the proof of the proposition from it.
Since $\Gibbs^u_c(f)$ is weak$^*$-compact, Lemma~\ref{l.newcriterion}, together with the
assumption that $f$ has $c$-mostly contracting center,
implies that there are $m \ge 1$ and $a<0$ such that
$$
\frac 1m \int_M \log \|Df^m\mid_{E_f^{cs}}\| \, d\mu < a
\text{ for every } \mu \in \Gibbs^u_c(f).
$$
Then, using \eqref{eq.convergencyGibbs} and the fact that the center-stable bundle $E^{cs}$
depends continuously on the diffeomorphism,
$$
\frac 1m \int_M \log \|Df_n^m\mid_{E_{f_n}^{cs}}\| \, d\mu_n < \frac{a}{2}
\text{ for any $\mu_n \in \Gibbs^u_c(f)$ and any large $n\ge 1$.}
$$
Invoking Lemma~\ref{l.newcriterion} once more, we see $f_n$ does have $c$-mostly contracting center
for every large $n$.
This reduces the proof of Proposition~\ref{p.robust} to proving the previous lemma:

\begin{proof}[Proof of Lemma~\ref{l.newcriterion}]
Let $g$ have $c$-mostly contracting center. By Proposition~\ref{p.contracting}, the center Lyapunov
exponents of any ergodic $c$-Gibbs $u$-state $\mu$ of $g$ are all negative.
So, for any $\mu$, there are $m_\mu \ge 1$ and $a_\mu<0$ such that
$$
\frac 1{m_\mu} \int_M  \log \|Dg^{m_\mu}\mid_{E^{cs}}\| \, d\mu < a_\mu
$$
Observe that the inequality remains true for any probability measure in an open neighborhood of $\mu$.
Thus, by compactness, we may take $m=m_\mu$ and $a=a_\mu$ independent of $\mu$.
This proves the `only if' part of the statement.

To prove the converse, assume there are $m \ge 1$ and $a<0$ such that
$$
\frac 1m \int_M \log \|Dg^{m} \mid_{E^{cs}}\| \, d\mu < a.
$$
for any ergodic $c$-Gibbs $u$ state $\mu$ of $g$. According to the theorem of Oseledets,
the largest center center-stable Lyapunov exponent coincides with
$$
\lim_n \frac{1}{nm} \int_M \log \|Dg^{nm}\mid_{E^{cs}}\| \, d\mu
$$
which, by subadditivity, is bounded above by
$$
\frac{1}{m} \int_M \log \|Dg^{m}\mid_{E^{cs}}\| \, d\mu < a <0.
$$
This shows that the center Lyapunov exponents of any $c$-Gibbs $u$-state are negative.
By Proposition~\ref{p.contracting}, it follows that $g$ has $c$-mostly contracting center.
\end{proof}
The proof of Proposition~\ref{p.robust} is complete.
\end{proof}

\section{$\TT^3$ diffeomorphisms derived from Anosov}\label{s.example1}

In the remaining of the paper we exhibit several examples of partially hyperbolic,
dynamically coherent diffeomorphisms satisfying all the assumptions of Theorem~\ref{main},
that is, factoring over Anosov and having $c$-mostly contracting center.

For the first type of example, let $A$ be a linear Anosov diffeomorphism on three dimensional torus,
with three positive eigenvalues $0<\kappa_1< \kappa_2<1<\kappa_3$, and denote by $E_1, E_2, E_3$ the
corresponding eigenspaces. We are going to view $A$ as a partially hyperbolic diffeomorphism of $\TT^3$
with $E^{ss}_A = E^1$, $E^c_A = E^2$, $E^{uu}_A = E^3$, and $E^{cs} = E^1\oplus E^2$.

We denote $\cD(A)$ the space of $C^1$ partially hyperbolic diffeomorphisms $f:\TT^3 \to \TT^3$ in the
isotopy of $A$, which we call \emph{derived from Anosov}.
This terminology goes back Smale~\cite{Sma67}, for $2$-dimensional maps.
Partially hyperbolic diffeomorphisms derived from Anosov were first studied by Ma\~n\'e~\cite{Man78}.
It is has been shown by Ures~\cite{Ure12}, and Viana, Yang~\cite[Theorem~3.6]{ViY17} that every
$f\in\cD(A)$ admits a unique measure of maximal entropy.

\begin{proposition}\label{p.example1}
Every $f\in \cD(A)$ is dynamically coherent, factors over Anosov, and has $c$-mostly contracting center.
Hence it satisfies the conclusions of Theorem~\ref{main}.

Furthermore, every $f\in\cD(A)$ has a unique measure of maximal $u$-entropy,
and it coincides with the measure $\mu_f$ of maximal entropy.
Moreover, its support is connected.
\end{proposition}

\begin{proof}
Dynamical coherence was proven by Potrie~\cite[Theorem~A.1]{Pot15}.
Moreover, Franks~\cite{Fra70} proved that there exists a continuous surjective map $\pi: \TT^3\to \TT^3$
such that $\pi \circ f=A\circ \pi$. This gives condition (H1).
Corollary~7.7 and Remark~7.8 in Potrie~\cite{Pot15} give that the semiconjugacy $\pi$ maps each
strong-unstable leaf of $f$ to an unstable leaf of $A$, as required in condition (H2).
Finally, Potrie~\cite[Theorem~7.10]{Pot15} also proved that the corresponding statement for center-stable
leaves, as in condition (H3).

At this point $f$ is known to satisfy all the assumptions of Theorem~\ref{main.maximal.measures} and thus
of Lemma~\ref{l.topologicalentropy}. According to that lemma, for any $c$-Gibbs $u$-state $\mu$,
$$
h_\mu(f) \ge h_\mu(f,\cF^{uu}) = h_{top}(A)=\log k_3.
$$
We also know, from Viana, Yang~\cite[Theorem~3.6]{ViY17}, that the push-forward map $\pi_*$ preserves the
entropy, and is a bijection restricted to the subsets of invariant ergodic probability measures with entropy
larger than $|\log k_1|$. These facts imply that $h_{\pi_*\mu}(A) = h_{top}(A)$, and so $\pi_*\mu$ coincides
with the (unique) measure of maximal entropy of $A$, namely, the Lebesgue measure $\nu$ on $\TT^3$.
By the injectivity of $\pi_*$, it also follows that $\mu$ is the unique measure of maximal entropy of
$f$, and the unique $c$-Gibbs $u$-state. We denote this measure by $\mu_f$ to highlight show its dependence
on the diffeomorphism.

It remains to show that the center Lyapunov exponent of $\mu_f$ is negative.
In fact, we claim that
\begin{equation}\label{eq.lambda2}
\lambda^c(\mu_f) \leq \log k_2 <0.
\end{equation}
This was proved by Ures~\cite[Theorem 5.1]{Ure12} in the special case of $C^2$ partially hyperbolic
diffeomorphisms. In fact, he used the stronger, so-called \emph{absolute}, version of partially hyperbolicity,
whereas here we always refer to the more general \emph{pointwise} version.
However, the only step where absolute partial hyperbolicity is used in his argument is for proving that the
unstable foliation is quasi-isometric, and that has has been proved to hold in the general pointwise case,
by Hammerlindl, Potrie~\cite[Section~3]{HaP14}. Thus, in order to finish the proof of \eqref{eq.lambda2}
we only have to remove the $C^2$ restriction.

For any $C^1$ element $f\in\cD(A)$, consider a sequence of $C^2$ diffeomorphisms $f_n\in\cD(A)$ converging
to $f$ in the $C^1$ topology. Let $\mu_n$ denote the measure of maximal entropy of each $f_n$.
We claim that $(\mu_n)_n$ converges to $\mu_f$. To prove this, let $\tilde\mu$ be any accumulation point.
Since these diffeomorphisms are away from homoclinic tangencies, it follows from Liao, Viana, Yang~\cite{LVY13}
that the entropy varies upper semi-continuously:
$$
h_{\tilde\mu}(f) \geq \limsup_n h_{\mu_n}(f_n) = h_{top}(A).
$$
Thus, by our previous arguments, $(\pi_f)_*(\tilde\mu)$ is the Lebesgue measure on $\TT^3$,
and $\tilde\mu$ is the unique maximal measure of $f$. This proves the claim.
Now, as the center bundle is one-dimensional, it follows that $\lambda^c(\mu_n) \to \lambda^c(\mu_f)$.
Since we already know that $\lambda^c(\mu_n) \le \log k_2$ for every $n$, this completes the proof of
\eqref{eq.lambda2} and of the proposition.

Let $l \ge 1$ be the number of connected components of the support of $\mu_f$.
Then, as in \eqref{eq.ergodic_decomposition}, the map $f^l$ has $l$ measures of maximal $u$-entropy.
On the other hand, it is clear that $f^l \in \cD(A^l)$, and so we may apply the previous arguments
to it. In particular, we get that $f^l$ has a unique measure of maximal $u$-entropy. Thus $l=1$.
\end{proof}

\section{Partially hyperbolic diffeomorphisms with circle fiber bundle}\label{s.example2}

In this section, we consider diffeomorphisms in the set $\SPH_1(M)$ of $C^2$ partially hyperbolic, accessible, dynamically coherent diffeomorphisms with 1-dimensional center such that the center foliation $\cF^c$ forms a circle bundle, and the quotient space $M_c=M/\cF^c$ is a topological torus.
This class of dynamical systems was studied previously by Ures, Viana, Yang~\cite{UVY} and
Hertz, Hertz, Tazhibi, Ures~\cite{HHTU12}.

Recall that a partially hyperbolic diffeomorphism is \emph{accessible} if any two points $x$, $y$
may be joined by a curve formed by finitely many arcs each of which is tangent to either the
strong-stable subbundle $E^s$ or the strong-unstable bundle $E^u$.
Accessibility is known to be a $C^1$ open and $C^r$ ($r \ge 1$) dense property for the partially hyperbolic diffeomorphisms with 1-dimensional center direction (see~\cite{BHHTU,Did03}).

\begin{proposition}\label{p.example2}
Every $f \in \SPH_1(M)$ factors over Anosov.
\end{proposition}

\begin{proof}
Let $f_c$ denote the map induced by any given $f\in \SPH_1(M)$ on the quotient space $M_c=M/\cF^c$,
and let $\pi: M \to M_c$ be the canonical quotient map. Then $\pi$ is a semi-conjugacy from $f$ to $f_c$. Moreover, $f_c$ is a topological Anosov homeomorphism, that is,
a globally hyperbolic homeomorphism, in the sense of \cite[Section 1.3]{Almost}
or~\cite[Section 2.2]{ViY13}. By a result of Hiraide~\cite{Hir90}, the map $f_c$ is conjugate to
a linear Anosov torus diffeomorphism $A$. Up to replacing $\pi$ with its composition with this
conjugacy, if necessary, it is no restriction to suppose that $f_c=A$. We do so from now on.
This gives condition (H1).

As before, let $\cF^c$, $\cF^{ss}$, $\cF^{uu}$ and $\cF^{cs}$ denote, respectively,
the center, strong-stable, strong-unstable and center-stable foliations of $f$.
We claim that every center leaf $\cF^c(x)$ intersects the strong-stable leaf $\cF^{ss}(x)$
only once. Indeed, suppose that there exists another point $y$ in the intersection.
The distance between $f^n(x)$ and $f^n(y)$ along the strong-stable leaf decreases exponentially.
Then, clearly, the same is true for the distance in the ambient manifold.
Since $f^n(x)$ and $f^n(y)$ belong to the same center leaf, and the center leaves are a
continuous family of $C^1$ embedded circles, this can only happen if the distance between
the points along the center leaf also decreases at the same rate. That is impossible because,
by domination, the contraction rates along center leaves are strictly weaker than along
strong-stable leaves. This contradiction proves the claim.
It follows that, under the map $\pi:M\to\TT^d$, the center-stable leaf $\cF^{cs}(x)$ projects down
to the stable manifold of $\pi(x)$ for the linear automorphism $A$.
Analogously, the strong-unstable leaf $\cF^{uu}(x)$ projects down to the unstable manifold of
$\pi(x)$ for $A$. Thus conditions (H2) and (H3) are also proved.
\end{proof}

\begin{remark}
The $C^2$ condition in the definition of $\SPH_1(M)$ was not used at all in the proof of the proposition.
Thus, the conclusion that $f$ factors over Anosov holds for every $C^1$ diffeomorphism that satisfies the
other conditions in the definition.
\end{remark}

Next, we discuss the $c$-mostly contracting center condition for this class of maps.

We begin by noting that, since each $\pi^{-1}(x_c), x_c \in M_c$ is a circle with
uniformly bounded length, and $f$ acts by homeomorphisms on those circles,
the projection $\pi$ preserves the topological entropy, and so $h_{top}(f)=h_{top}(A)$.
Thus, using Lemma~\ref{l.topologicalentropy}, we get that every $c$-Gibbs $u$-state
$\mu$ is a measure of maximal entropy for any $f\in\SPH_1(M)$:
$$
h_\mu(f)\geq h_\mu(f,\cF^{uu})=h_{top}(A)=h_{top}(f).
$$
This also proves that any measure of maximal $u$-entropy measure is also a measure of
maximal entropy.

Let us also recall the dichotomy proved in \cite{HHTU12}: for any $f\in SPH_1$,
\begin{enumerate}
\item[(a)] either $f$ is conjugate to a rotation extension of an Anosov diffeomorphism,
in which case it has a unique measure of maximal entropy, and this measure has full support
and vanishing center exponent;
\item[(b)] or $f$ admits some hyperbolic periodiC point, and has finitely many ergodic
measures of maximal entropy, all of which have non-vanishing center exponent.
\end{enumerate}

In case (a), the previous discussion implies that there is a unique measure of maximal $u$-entropy,
which coincides with the unique measure of maximal $u$-entropy. The next proposition deals
with the other case:

\begin{proposition}\label{p.SPH}
If $f\in\SPH_1(M)$ has some hyperbolic periodic point then, then it has $c$-mostly
contracting center. Moreover, an ergodic measure of $f$ is a maximal $u$ entropy measure
if and only if it is an ergodic maximal measure of $f$ with negative center exponent.
\end{proposition}

\begin{proof}
This now follows directly from \cite[Lemma~5.1]{UVY}, here it was proved that the space of
$c$-Gibbs $u$-states (called  $\nu$-Gibbs $u$-states and denoted as $\Gibb^u_\nu(f)$ in that
paper) coincides with the the finite-dimensional simplex (denoted as $\MM^-(f)$ in that paper)
generated by the ergodic measures of maximal entropy and negative center Lyapunov exponent.
\end{proof}

Observe that although this construction is done in the $C^2$ category, it follows from
Proposition~\ref{p.robust} that the conclusion remains true in a whole $C^1$-neighborhood.

It is also worth mentioning the special case of partially hyperbolic diffeomorphisms on
$3$-dimensional nilmanifolds $M$ other than $\TT^3$, which were also studied in \cite{UVY}.
It was proved by Hammerlindl, Potrie~\cite[Propositions 1.9 and 6.4]{HaP14} that any
such diffeomorphism is in $\SPH_1(M)$ and, in addition, admits a unique compact, invariant,
$u$-minimal subset. Thus, on 3-nilmanifolds the measure of maximal $u$-entropy is always
unique, and its support is necessarily connected (even in case (b) above).

\section{Partially volume expanding topological solenoids}\label{s.example4}

Our next family of examples, which is a variation of the the classical solenoid construction
of Smale~\cite{Sma67}, was studied previously by Bonatti, Li, D. Yang~\cite{BLY13} and
Gan, Li, Viana, J. Yang~\cite{GLVY}. Formally speaking, these are just embeddings,
rather than diffeomorphisms, and so our previous results do not apply immediately to them.
However, it is clear that the arguments we presented previously extend to the
embedding setting, and these applications are worthwhile mentioning here.
Indeed, these examples may exhibit homoclinic tangencies and infinitely many
coexisting sources (see~\cite{BLY13}),
which was not the case in Sections~\ref{s.example1} and~\ref{s.example2}.
The reason this is now possible is that these examples exhibit no domination
inside the center-stable bundle.

Let $D$ be the 2-dimensional disk.
By a \emph{Smale solenoid} we mean an embedding $g_0:M \to M$ of the solid torus
$M = S^1 \times D$ of the form
$$
g_0(\theta,x)=(k\theta \mod 1, ax + b(\theta)),
$$
where $k \ge 3$ and $a\in (k^{-1},1)$ are independent of $\theta$, and
$b:S^1 \to D$ is a suitable $C^1$ map.

Now consider any $C^1$ embedding $f_0: M\to M$ of the form
\begin{equation}\label{eq.topologicalsolnoid}
    f_0(\theta, x) = (k \theta \mod 1, h_\theta(x))
\end{equation}
here $h_\theta$ is such that $\|Dh_\theta\|$ and $\|Dh^{-1}_\theta\|$
are both strictly less than $k$ at every point.
For every embedding $f:M\to M$ in a $C^1$ neighborhood of $f_0$, we denote by $\Lambda(f)$
the maximal invariant  set, that is, $\Lambda(f) = \cap_{n>0} f^n(M)$.
The definition implies that $\Lambda(f)$ is a $u$-saturated set.

While the Smale solenoid was introduced as a model for uniformly hyperbolic
dynamics, it is clearly possible to pick $f_0$ satisfying these assumptions in such a way
that $\Lambda(f_0)$ is not a uniformly hyperbolic set.
On the other hand, it was shown in \cite[Lemma~6.2]{GLVY} that $\Lambda(f_0)$ is a partially hyperbolic set
for some iterate $f_0^N$ (in what follows we take $N$ to be $1$).
Since partial hyperbolicity is a robust property, it follows that $\Lambda(f)$
is still partially hyperbolic for every $f$ in a $C^1$-neighborhood.
In addition, it follows from the stability theorem of
Hirsch, Pugh, Shub~\cite{HPS77} that $f$ is dynamically coherent and, in fact,
its center foliation is topologically conjugate to the center foliation of $f_0$,
that is, to the vertical fibration $\{\{\theta\}\times D: \theta \in S^1\}$.

\begin{proposition}\label{p.topologicalsolenoid}
Every $f$ in a $C^1$ neighborhood of $f_0$ factors over Anosov,
restricted to the maximal invariant set $\Lambda(f)$.
\end{proposition}

\begin{proof}
First we deal with the case $f=f_0$ and then we explain how the arguments can be adapted to any
$C^1$-small perturbation.

Let $\pi^{cs}:M \to S^1$ be the canonical projection the projection $\pi^{cs}(\theta,x)=\theta$.
For each $(\theta,x)\in \Lambda(f_0)$ and $n \in \ZZ$, define $\theta_n=\pi^{cs}(f_0^n(\theta,x))$.
The sequence $(\theta_n)_{n\in\ZZ}$ is an orbit for the circle map $x \to kx \mod 1$.
Moreover, since $g_0: \Lambda(g_0) \to \Lambda(g_0)$ is the natural extension of that map,
there exists exactly one point $y \in D$ such that $(\theta,y) \in \Lambda(g_0)$ and
$\theta_n=\pi^{cs}(g_0^n(\theta, y))$ for every $n \in \ZZ$.

Let $\pi:\Lambda(f_0) \to \Lambda(g_0)$ be the map $\pi: (\theta,x) = (\theta,y)$ defined
in this way. This map $\pi$ is continuous and surjective, and it is a semiconjugacy between
$f_0$ and $g_0$. Moreover, by construction, it maps each local center-stable ``leaf''
$\Lambda(f_0) \cap \left(\{\theta\} \times D\right)$ of $f_0$ to the center-stable ``leaf''
$\Lambda(g_0) \cap \left(\{\theta\} \times D\right)$ of $g_0$.
Thus we proved conditions (H1) and (H3) for $f=f_0$, in versions suitable for the present
setting.

We are left to verifying condition (H2).
Let $(\theta,x)$ and $(\theta^\prime,x^\prime)$ be any two points of $\Lambda(f_0)$ in
the same local unstable manifold. Let $y^\prime$ be the point of $D$ such that $\pi(\theta^\prime,x^\prime)=(\theta^\prime,y^\prime)$.
For every $n<0$, the points $f_0^n(\theta,x)$ and $f_0^n(\theta^\prime,x^\prime)$
are contained in a cylinder $I_n \times D$, where $I_n$ is an interval whose length
goes to zero when $n\to-\infty$.
Since $f_0^n(\theta,x)$ and $g_0^n(\theta,y)$ belong to the same vertical disk
$\{\theta_n\} \times D$ for every $n$, and the same holds for
$f_0^n(\theta^\prime,x^\prime)$ and $g_0^n(\theta^\prime,y^\prime)$,
we conclude that $g_0^n(\theta,y)$ and $g_0^n(\theta^\prime,y^\prime)$ are also
contained in the cylinder $I_n \times D$. This can only occur if $(\theta,y)$
and $(\theta^\prime,y^\prime)$ belong to the same local strong-unstable leaf of $g_0$.
Thus, we proved that $\pi$ maps every strong-unstable leaf of $f_0$ inside an strong-unstable
leaf of $g_0$, as we wanted to prove.

Now let us extend this construction to a $C^1$-neighborhood of $f$.
The stability theorem of \cite{HPS77} gives that if $f$ is $C^1$-close to $f_0$
then its center foliation $\cF^c_f$ is conjugate to the center foliation of $f_0$
(that is, the vertical lamination $\left\{\{\theta\}\times D: \theta \in S^1\right\}$)
by a homeomorphism close to the identity.
Let $\pi^{cs}_f: M \to M/\cF^c_f$ be the quotient map: the quotient space is
a topological circle, and the induced quotient map $f_{cs}$ is conjugate to the
circle map $x \mapsto kx \mod 1$ by some homeomorphism $\psi:M/\cF^c_f \to S^1$.

For every $(\theta,x)\in \Lambda(f)$, consider the sequence $(\theta^f_n)_n$ on $S^1$ defined by
$$
\theta_n^f = \psi\left(\pi^{cs}_f\left(f^n(\theta,x)\right)\right).
$$
As previously, this sequence identifies a unique point $(\theta_0,y) \in \Lambda(g_0)$
such that
$$
\theta_n=\pi^{cs}\left(g_0^n(\theta_0,y)\right)
\text{ for any } n\in \ZZ.
$$
Let $\pi_f:\Lambda(f) \to \Lambda(g_0)$ be the map given by
$\pi_f(\theta,x)=(\theta_0,y)$. By the same arguments as before,
$\pi_f$ is a semiconjugacy between $f \mid_{\Lambda(f)}$ and $g_0 \mid_{\Lambda(g_0)}$,
verifying the conditions (H1), (H2), (H3).
\end{proof}

Following Gan, Li, Viana, Yang~\cite{GLVY}, we say that a partially hyperbolic diffeomorphism $f:M \to M$ is \emph{partially volume expanding} if
\begin{equation}
   |\det Df(x) \mid_H| > 1
\end{equation}
for any hyperplane $H$ of $T_xM$ that contains $E^{uu}(x)$.
This is a $C^1$ open property, and a lower bound $b > 1$ for the Jacobian may be
chosen uniformly on a neighborhood.

By \cite[Lemma~6.3]{GLVY}, the map $f_0$ above is partially volume expanding and,
thus, so is every $f$ in a $C^1$-neighborhood of $f_0$.
Let $\mu$ be an ergodic invariant probability measure for $f$.
We denote by $\lambda^u(\mu,f)$ its unstable Lyapunov exponent,
and by $\lambda^c_1(\mu,f) \leq \lambda^c_2(\mu,f)$ the two (strictly smaller)
center Lyapunov exponents.
By~\cite[formula (7)]{GLVY}, partial volume expansion implies that there is a
constant $c>0$, independent of $f$ in a neighborhood of $f_0$, such that
\begin{equation}\label{eq.partiallyve}
\lambda^u(\mu,f)+\lambda^c_1(\mu,f) > c > 0.
\end{equation}

\begin{proposition}
Every embedding $f$ in a $C^1$-neighborhood of $f_0$ has $c$-mostly contracting
center on the maximal invariant set $\Lambda(f)$.
Moreover, the space of $c$-Gibbs $u$-states of $f \mid_{\Lambda(f)}$ coincides
with the space of measures of maximal entropy.
\end{proposition}

\begin{proof}
Observe that the unstable Lyapunov exponent of $f_0$ relative to any invariant
measure is equal to $\log k$. Since $E^u$ is $1$-dimensional, this exponent
varies continuously with the diffeomorphism. In particular, there exists a neighborhood of $f_0$ such that $\lambda^u(\cdot,f) > \log k - c$ for any $f$
in that neighborhood. Let $\mu$ be any ergodic $c$-Gibbs $u$-state $\mu$ of $f$.
Lemma~\ref{l.topologicalentropy} gives that
\begin{equation}\label{eq.L31}
h_\mu(f\mid_{\Lambda(f)})
\geq h_\mu(f\mid_{\Lambda(f)},\cF^{uu})=h(g_0)=\log k.
\end{equation}
Combining this with \eqref{eq.partiallyve}, we find that
\begin{equation*}
\lambda^c_1(\mu,f)
> c-\lambda^u(\mu,f)
> -\log k.
\end{equation*}
If $\lambda^c_2(\mu,f)$ were non-negative then the
Ruelle inequality (Ruelle~\cite{Rue78}) for $f^{-1}$ would yield
$$
h_\mu(f\mid_{\Lambda(f)})
= h_\mu(f^{-1}\mid_{\Lambda(f)})
\leq -\lambda^c_1(\mu,f)
< \log k,
$$
and that would contradict \eqref{eq.L31}.
Thus $\lambda^c_2(\mu,f) <  0$ for any ergodic $c$-Gibbs $u$-state $\mu$,
which means that $f$ has $c$-mostly contracting center.

Let $\mu_{max}$ be any ergodic measure of maximal entropy and $\mu$ be any
ergodic $c$-Gibbs $u$-state $\mu$ of $f$. Then
\begin{equation}\label{eq.L32}
h_{\mu_{max}}(f\mid_{\Lambda(f)})\geq h_\mu(f\mid_{\Lambda(f)})\geq \log k.
\end{equation}
By the previous argument, starting from \eqref{eq.L32} rather than \eqref{eq.L31},
we get that both center Lyapunov exponents of $\mu_{max}$ are negative,
and so $\mu_{max}$ is a hyperbolic measure.
By Proposition~\ref{p.uentropyandentropy} and Lemma~\ref{l.topologicalentropy},
it follows that
$$
h_{\mu_{max}}(f\mid_{\Lambda(f)})
= h_{\mu_{max}}(f\mid_{\Lambda(f)}, \cF^{uu})
\leq \log k.
$$
This has several consequences. To begin with, we get that $\mu_{max}$ is a measure
of maximal $u$-entropy. In other words, we have shown that every ergodic measure
of maximal entropy is also a measure of maximal $u$-entropy. Since these are convex
spaces whose extreme points are ergodic measures, it follows that the space of
measures of maximal entropy is contained in the space of measures of maximal
$u$-entropy, which (by Theorem~\ref{main.maximal.measures}) is known to coincide
with the space of $c$-Gibbs $u$-states.
Another consequence of the previous argument is that the topological entropy
$h(f \mid_{\Lambda(f)}) = h_{\mu_{max}}(f \mid_{\Lambda(f)})$ is equal to $\log k$.
Hence, using Lemma~\ref{l.topologicalentropy}, we also get that every $c$-Gibbs
$u$-state is a measure of maximal entropy.
\end{proof}

\section{Measures of maximal $u$-entropy which are not of maximal entropy}\label{s.example5}

The next examples are a modification of the construction in Section~\ref{s.example4}
the main novel feature being that, for the first time in our list of examples,
measures of maximal $u$-entropy may fail to be of maximal entropy.
This phenomenon cannot occur in the setting of uniformly hyperbolic systems.
These counterexamples are constructed on certain horseshoes which arise from homoclinic
tangencies taking place inside suitable (periodic) center-stable leaves.
The presence of such horseshoes \emph{with large entropy} is made possible, in part,
by the fact that we no longer have partial volume expansiveness, as we did in the
previous section.

\subsection{Modified topological solenoid}\label{s.modified_solenoid}

Fix $k\geq 3$, $\alpha \in (0,{\pi}/{2})$ and $a\in({k}^{-1},1)$.
Let $\phi: D \to \Int(D) $ be the embedding defined by
\begin{equation}\label{eq.contractionmap}
\phi(x)=a\circ R_{\alpha}(x)
\end{equation}
where $R_{\alpha}$ denotes the rotation of angle $\alpha$ about the origin.

Let $\Diff^1_0(D, \Int(D))$ be the space of orientation preserving $C^1$-embeddings
$D \to \Int(D)$. Every element of this space is isotopic to $\phi$.
Plykin~\cite{Ply80} built a non-empty open set $P \subset \Diff^1_0(D, \Int(D))$
such that for every $\phi \in P$ the chain recurrent set consists of the union of
a non-trivial uniformly hyperbolic attractor $A(\phi)$ with a finite set of
periodic sources. Consider any $\phi_0 \in P$ such that $A(\phi_0)$ contains a
fixed saddle-point $p$ with eigenvalues $k^u > k$ and $k^s < {1}/{k}$ such that
$k^u k^s >1$. Let $(\psi_t)_{t\in[-1,1]}$ be a smooth path in $\Diff^1_0(D,\Int(D))$
such that $\psi_0=\phi_0$ and $\psi_t= \phi$ for every $t$ close to the
endpoints $-1$ and $1$. Denote
$$
K=\max\{\|D\psi_t(x)\|: x\in D\}.
$$
Let $\beta: S^1\to S^1$ be a uniformly expanding map of degree $k$ such that
$\dot \beta (\theta) > K$ for every $\theta$ in a neighborhood $[-\epsilon_0,\epsilon_0]$
of the fixed point $0$, and let $\nu_\beta$ be its (unique) measure of maximal entropy.
Since $\nu_\beta$ has no atoms, we may find $\epsilon\in(0,\epsilon_0)$ and $c>0$ such that
\begin{equation}\label{eq.negativecenterexponents}
    \nu_\beta([-\epsilon,\epsilon]) \log K + \nu_\beta(S^1 \setminus [-\epsilon,\epsilon]) \log a<-c<0.
\end{equation}
Let $z \notin [-\epsilon,\epsilon]$ be another fixed point of $\beta$,
and $b:S^1 \to D$ be a $C^1$ map such that $b(0)=b(z)=0$.
Define $f_1: M\to M$ by $f_1(\theta, x)=(\beta(\theta), h_{\theta}(x))$ where
\begin{equation}\label{eq.htheta}
h_\theta=\left\{\begin{array}{ll}
\psi_{{\theta}/{\epsilon}} + b(\theta) & \text{for } \theta \in [-\epsilon, \epsilon] \\
\phi + b(\theta) &\text{for } \theta \notin [-\epsilon, \epsilon].
\end{array}\right.
\end{equation}

Observe that $\|Dh_\theta\| = \|D\psi_{\theta/\epsilon}\| < K < \dot \beta(\theta)$
for $\theta \in [-\epsilon,\epsilon]$, and $\|Dh_\theta\| = a < 1 < \dot \beta(\theta)$
for $\theta \notin [-\epsilon,\epsilon]$.
Thus
$$
\|D h_{\theta}(\cdot)\| < \dot \beta(\theta)
\text{ at every $\theta\in S^1$}
$$
which implies that $f_1$ is partially hyperbolic on its maximal invariant set $\Lambda(f_1)$
The center foliation is just that vertical fibration $\{\{\theta\}\times D: \theta \in S^1\}$.

Then every embedding $f:M\to M$ in a $C^1$ neighborhood of $f_1$ is partially hyperbolic on
$\Lambda(f)$ and, using the stability theorem of \cite{HPS77}, is dynamically coherent.
Now, precisely as in Proposition~\ref{p.topologicalsolenoid}, we get that $f \mid_{\Lambda(f)}$
is semiconjugate to the classical Smale solenoid $g_0:M\to M$  (condition (H1)),
by a semiconjugacy that maps each strong-unstable leaf of $f$ inside $\Lambda(f)$ to a strong-unstable leaf
of $g_0$ inside $\Lambda(g_0)$ (condition (H2)),
and each center-stable ``leaf'' $\Lambda(f) \cap \cF^{cs}_{\loc}(\cdot)$ of $f$ to
a center-stable ``leaf'' $\Lambda(g_0) \cap (\{\theta\}\times D)$ of $g_0$
(condition (H3)).

\subsection{$c$-mostly contracting center}

Next we prove that $f_1$ has $c$-mostly contracting center. In view of Proposition~\ref{p.robust},
it follows that the same is true for every diffeomorphism $f$ in a $C^1$-neighborhood.

Let $\pi^{cs}:M\to S^1$ denote the projection to the first coordinate.
Recall that $\nu_{\beta}$ denotes the measure of maximal entropy of $\beta:S^1 \to S^1$.
Since $\beta$ is uniformly expanding of degree $k$, it is conjugate to the map
$x \mapsto kx \mod 1$ by some homeomorphism $\pi_\beta:S^1 \to S^1$.

\begin{lemma}\label{l.quotienttocircle}
$(\pi^{cs})_*\mu =\nu_{\beta}$ for any $c$-Gibbs $u$-state $\mu$ of $f_1$.
\end{lemma}

\begin{proof}
By properties (H1) and (H3) for $f_1$, there exists a semiconjugacy $\pi_{f_1}$
between $f_1\mid_{\Lambda(f_1)}$ to $g\mid_{\Lambda(g)}$ such that
\begin{equation}\label{eq.communication}
\pi_\beta \circ \pi^{cs}=  \pi^{cs} \circ \pi_{f_1}.
\end{equation}
By Theorem~\ref{main.maximal.measures}, $(\pi_{f_1})_*\mu$ is the (unique) measure
of maximal entropy of $g_0$, and so $(\pi^{cs}\circ \pi_{f_1})_*\mu $ is the measure
of maximal entropy of $x \mapsto k x \mod 1$, that is, the Lebesgue measure on $S^1$.
By \eqref{eq.communication}, it follows that $(\pi_\beta \circ \pi^{cs})_*\mu$ is
the Lebesgue measure on $S^1$, that is, $(\pi^{cs})_*\mu$ is the measure $\nu_\beta$
of maximal entropy of $\beta$.
\end{proof}

\begin{lemma}
The map $f_1:M\to M$ has $c$-mostly contracting center on the maximal invariant set $\Lambda(f_1)$.
\end{lemma}

\begin{proof}
The definition \eqref{eq.htheta}
gives that
$$
\begin{aligned}
\|Df_1\mid_{E^{cs}}\| & = \|D h_\theta\| < K,
\text{ for $\theta \in [-\epsilon,\epsilon]$, and}\\
\|Df_1\mid_{E^{cs}}\| & = \|D\phi\| = a
\text{ for $\theta \notin [-\epsilon,\epsilon]$.}
\end{aligned}
$$
Thus, using \eqref{eq.negativecenterexponents},
$$
\int_{S^1} \log \left(\sup_{\{\theta\} \times D} \log \|Df_1\mid_{E^{cs}}\|\right) \, d \nu_{\beta}(\theta)\\
\leq K \nu([-\epsilon,\epsilon]) + a \nu(S^1\setminus [-\epsilon,\epsilon]) < 0.
$$
By Lemma~\ref{l.quotienttocircle}, this implies that
$\int_M \log \|Df_1\mid_{E^{cs}}\| \, d\mu<0$
for any ergodic $c$-Gibbs $u$-state $\mu$ of $f_1$.
\end{proof}

Recall that $z \notin [-\epsilon,\epsilon]$ is a fixed point of the map $\beta$.
It follows from the construction that $(z,0)$ is a fixed point of $f_1$ with exactly
$2 = \dim E^{cs}$ contracting eigenvalues. Let $q(f)$ denote its hyperbolic continuation
for nearby maps $f$. Then $\{q(f)\}$ is a skeleton for $f \mid_{\Lambda(f)}$,
and so Theorem~\ref{main.skeleton} yields:

\begin{corollary}\label{c.uniqueness}
Every $f$ in a $C^1$ neighborhood of $f_1$ admits a unique $c$-Gibbs $u$-state $\mu_f$,
whose support coincides with the homoclinic class of $q(f)$, and also with the closure
of $W^{u}(q(f))$.
\end{corollary}

By Theorem~\ref{main.maximal.measures}, we have that $\mu_f$ is the unique measure of
maximal $u$-entropy. By Proposition~\ref{p.uentropyandentropy}, the entropy and the
$u$-entropy of $\mu_f$ are equal. Then, by Lemma~\ref{l.topologicalentropy}, they
also coincinde with the topological entropy of the solenoid $g_0 \mid_{\Lambda(g_0)}$,
which is equal to $\log k$. In the remainder of this section, we are going to prove that
\begin{equation}\label{eq.done_at_last}
\log k < h (f \mid _{\supp\mu_f}),
\end{equation}
and so $\mu_f$ is not a measure of maximal entropy, for typical maps $f$ in the neighborhood
of $f_1$.

In a nutshell, we are going to identify a suitable periodic point $r(f)$ close to the hyperbolic
continuation $p(f)$ of the fixed point $(0,p)$, and to show that $r(f)$ exhibits homoclinic
tangencies. The proof of the latter uses the fact that the homoclinic class of $p(f)$ generically
coincides with the homoclinic class of the other fixed point $q(f)$.
This is important because, unlike $q(f)$, the fixed point $p(f)$ has a domination property along
the center-stable which prevents the existence of homoclinic tangencies.
Then a result of Newhouse~\cite{New78a} allows us to deduce that the unfolding of such tangencies
generates (small) uniformly hyperbolic sets with large entropy. The estimates involve the
eigenvalues of the periodic point $r(f)$, which we choose close to those of $p(f)$.

\subsection{Chain recurrence classes and homoclinic classes}\label{s.chain_recurrence}

To implement this approach, we start by recalling the notion of chain recurrent class.
See \cite[Chapter~10]{Beyond} for more information.

A point $x\in M$ is \emph{chain recurrent} if for any $\epsilon>0$ there exists a finite
$\epsilon$-pseudo orbit (with at least two points) starting and ending at $x$.
Two chain recurrent points $x$ and $y$ are \emph{chain recurrent equivalent} if for any
$\epsilon>0$ there exists a finite closed $\epsilon$-pseudo orbit containing both of them.
This is an equivalence relation, and its equivalence classes are called \emph{chain recurrent classes}.

\begin{lemma}\label{l.uniquequasiattractor}
For any embedding $f:M\to M$ in a $C^1$-neighborhood of $f_1$, the fixed points $p(f)$ and
$q(f)$ belong to the same chain recurrent class.
\end{lemma}

\begin{proof}
By construction, every strong-unstable leaf inside $\Lambda(f)$ intersects every center-stable
leaf. In particular, $\cF^{uu}(p(f))$ intersects $\cF^{cs}(q(f))$.
Recall that $q(f_1)=(z, 0)$ has two contracting eigenvalues in the center-stable direction.
Thus, our construction also gives that $\cF^{cs}(q(f))$ is contained in the stable manifold
$W^s(q(f))$, as long as $f$ is close enough to $f_1$.
Thus, there is a heteroclinic intersection from $p(f)$ to $q(f)$ and, in particular,
for any $\epsilon>0$ there is an $\epsilon$-pseudo orbit from $p(f)$ to $q(f)$.

It is also true that $\cF^{uu}(q(f))$ intersects $\cF^{cs}(p(f))$ at some point $w$,
but we cannot use the same argument as before because $(p,0) = p(f_1)$ has only one
contracting eigenvalue.
We bypass this by using the fact that $p(f)$ is contained in a uniformly hyperbolic
attractor $A(f)$ of $f \mid_{\cF^{cs}(p(f))}$, given by the hyperbolic continuation of the
Plykin attractor $A(\phi_0)$ at the beginning of Section~\ref{s.modified_solenoid}.
By construction~\cite{Ply80}, the Plykin attractor $A(\phi_0)$ is a transitive set,
and it basin contains the whole disk $D$ minus a finite set of periodic repellers.
By stability, these properties still hold for $A(f)$ as long as $f$ is close enough
to $f_1$. Since $w$ is clearly not a periodic point, as it is contained in the
unstable manifold of $q(f)$, the second property implies that $w$ is contained in the
basin of $A(f)$. This means that we also have a heteroclinic intersection from $q(f)$
to the transitive set $A(f)$, which contains $p(f)$.
In particular, for any $\epsilon>0$ there is an $\epsilon$-pseudo orbit from $q(f)$
to $p(f)$.

The claim follows directly from these two conclusions.
\end{proof}

We denote by $\C(f)$ the chain recurrence class in Lemma~\ref{l.uniquequasiattractor}.
It follows immediately from the definitions that it contains the homoclinic classes
$H(p(f),f)$ and $H(q(f),f)$ of both fixed points.
By a result of Bonatti, Crovisier~\cite{BoCr04},
\begin{equation}\label{eq.generic_equal}
H(p(f),f) = \C(f) = H(q(f),f)
\end{equation}
for every $f$ in a $C^1$-residual subset $\cR$ of some $C^1$-neighborhood $\cU$ of $f_1$.

\subsection{Large entropy created from a homoclinic tangency}

We are now going to explain how to create uniformly hyperbolic sets with large entropy,
by unfolding homoclinic tangencies inside the chain recurrence class $\C(f)$.
For this, we need to recall the notion of domination (see \cite[Chapter~7]{Beyond}).

Let $\Lambda$ be an $f$-invariant compact set, and $G\subset T_\Lambda M$ be a continuous
subbundle of the tangent bundle over $\Lambda$ invariant under the derivative $Df$.
Given any integer $N\ge 1$, we say that $(f,\Lambda,G)$ has an \emph{$N$-dominated splitting} if
there is a direct sum decomposition $G=E\oplus F$ into two continuous subbundles invariant
under $Df$ such that
$$
\|Df^n u\| \leq \frac1 2 \|Df^n v\|
\text{ for any $n\ge N$ and any unit vectors $u\in E$ and $v\in F$}.
$$
This is a closed property, in the following sense.
Suppose that there is an $N$-dominated splitting $G_n = E_n \oplus F_n$ for $(f_n,\Lambda_n,G_n)$,
for each $n\ge 1$, such that the dimensions of $E_n$ and $F_n$ are independent of $n$.
Suppose also that $f_n\to f$ in the $C^1$-toplogy, $\Lambda_n\to\Lambda$ in the
Hausdorff topology, and $G_n \to G$ uniformly. Then there is an $N$-dominated splitting
for $(f,\Lambda,G)$.

The next step is to create a homoclinic tangency associated to a periodic point
$r(f) \in H(p(f),f)$. This will be done by means of the following proposition, which is
contained in the combination of Gourmelon~\cite[Theorem~3.1]{Gou10} and
Gourmelon~\cite[Theorem~8]{Gou16}. See Buzzi, Crovisier, Fisher~\cite[Theorem~3.9]{BCF18}
for an analogous statement, and Pujals, Sambarino~\cite{PS00a} and Wen~\cite{Wen02} for
previous results in this direction.

$\cU$ denotes a neighborhood of $f_1$ as introduced at the end of Section~\ref{s.chain_recurrence}.

\begin{theorem}[Gourmelon]\label{t.tangency}
For any $\epsilon> 0$, there exist $N\geq 1$ and $T\geq 1$ such that the following holds.
Consider any $f\in \cU$ admitting a periodic saddle $O$ with $1$ contracting eigenvalue
and period greater than $T$, such that the restriction of $Df$ to the center-stable
space $E^{cs}_O$ does not have an $N$-dominated splitting.

Then there exists an $\epsilon$-perturbation $g$ of $f$ in the $C^1$ topology such that
\begin{itemize}
    \item $g$ coincides with $f$ outside an arbitrarily small neighborhood of $O$ and
    \item $g$ preserves $\Orb(O)$ and the derivatives of $f$ and $g$ coincide along $\Orb(O)$,
\end{itemize}
and there exists $O_0 \in \Orb(O)$ such that $W^s(O_0)$ and $W^u(O_0)$ have a tangency point
$Z \in \cF^{cs}(O_0)$ whose orbit is contained in an arbitrarily small neighborhood of $O$.
Moreover, if $O$ is homoclinically related to a periodic point $P$ of f, then $g$ may be
chosen such that $O$ remains homoclinically related to $P$ for $g$.
\end{theorem}

Let $f\in \cU$. Observe that the homoclinic class of $p(f)$ is non-trivial, since it contains
the uniformly hyperbolic transitive set $A(f)$. Then one can find a sequence $(p_{n})_n$ of
hyperbolic periodic points whose orbits $\Orb(p_n)$ converge to $H(p(f),f)$ in the Hausdorff
topology, and which are homoclinically related to $p(f)$; the latter also means that each $p_n$
has exactly one contracting eigenvalue.

The fact that $p_n$ and $p(f)$ are homoclinically related ensures that we can find a
uniformly hyperbolic transitive set $H_n$ containing both their orbits.
Then, by the shadowing lemma, we can find a periodic point $q_n \in H_n$ homoclinically related
to both $p_n$ and $p(f)$, whose orbit spends most of the time near $p(f)$ and yet gets denser
in $H_n$ as $n$ increases. More precisely, as $n\to\infty$,
\begin{itemize}
    \item the invariant probability measure $\mu_n$ supported on $\Orb(q_n)$ converges to the
    Dirac mass $\delta_{p(f)}$
    \item and the Hausdorff distance between $H_n$ and $\Orb(q_n)$ goes to zero.
\end{itemize}
The latter implies that the sequence $\Orb(q_n)$ also converges to the whole homoclinic class
in the Hausdorff topology. For future reference, let $\kappa_n \ge 1$ denote the period of
each $q_n$, which converges to infinity as $n\to\infty$.

Finally, using the assumptions on the eigenvalues of $p(f_1)=(0,p)$, we can now prove that
the unfolding of the homoclinic tangencies given by Theorem~\ref{t.tangency} does yield
uniformly hyperbolic sets with large entropy, as stated in \eqref{eq.done_at_last}:

\begin{proposition}\label{p.largeentropy}
There is a $C^1$ open and dense subset $\cV$ of $\cU$ such that for any $f\in \cV$,
there exists a periodic point $r(f)$ homoclinically related to $p(f)$,
and there exists a uniformly hyperbolic set contained in the center-stable leaf of
$\Orb(r(f))$ whose topological entropy is larger than $\log k$.
\end{proposition}

\begin{proof}
It is clear that the property in the conclusion is open, and so we only need to prove that it
is dense. Take $f$ in the residual set $\cR$ where the identity \eqref{eq.generic_equal} holds.
Since both eigenvalues of $q(f)$ along the center-stable direction are non-real, the continuous
bundle $G = E^{cs}\mid_{H(p(f),f)}$ admits no dominated splitting.
By Bochi-Bonatti~\cite[Theorem~3]{BoB12}, it follows that for any $\epsilon>0$,
there is $N \ge 1$, such that for every $n\geq N$ there is an $\epsilon$-perturbation $f_n$
of $f$ in the $C^1$ topology such that
\begin{itemize}
    \item $f_n$ coincides with $f$ outside an arbitrarily small neighbourhood of $\Orb(q_n)$,
    \item $f_n$ preserves $\Orb(q_n)$ and
    \item the eigenvalues $k^s_n$ and $k^u_n$ of $Df_n^{\kappa_n}$ along the center-stable
    space of $q_n$ satisfy
\end{itemize}
\begin{equation}\label{eq.kkk}
\frac{1}{\kappa_n} \log k^s_n \to \log k^s
\text{ and }
\frac{1}{\kappa_n} \log k^u_n\to \log k^u.
\end{equation}
Keep in mind that $k^s < 1/k < k < k^u_n$. Also, by the version of Franks' lemma in Gourmelon~\cite{Gou16},
this perturbation can be made in such a way that the periodic point $q_n$ of $f_n$ is still homoclinically related to $p(f_n)$.

Now, since $\Orb(q_n)$ converges to $H(p(f),f)$ in the Hausdorff topology and,
as we just observed, the center-stable bundle $E^{cs}$ has no domination over the homoclinic
class, we get that any domination of $E^{cs}$ over the orbit of $q_n$ has to become
arbitrarily weak as $n\to\infty$. More precisely, for any $N \ge 1$ and any $n$ large enough,
the bundle $E^{cs} \mid_{\Orb(q_n)}$ for the map $f_n$ has no $N$-dominated splitting.
So we may apply Theorem~\ref{t.tangency} to obtain an embedding $g_n$ near $f_n$,
such that $g_n$ has a homoclinic tangency on the center-stable leaf of $\Orb(q_n)$.

Finally, by a result of Newhouse~\cite{New78a}, a suitable unfolding of the homoclinic
tangency yields a small uniformly hyperbolic transitive set with topological entropy close to
$$
\min\left\{-\frac{1}{\kappa_n}\log k^s_n,\frac{1}{\kappa_n} \log k^u_n \right\}.
$$
By \eqref{eq.kkk}, it follows that the topological entropy is larger than $\log k$
if $n$ is large.
\end{proof}

\appendix

\section{Ergodic results for $C^1$ partially hyperbolic diffeomorphisms}

Results in the ergodic theory of non-uniformly hyperbolic diffeomorphims are often stated
under the assumption that the derivative is H\"older. Here we prove certain
extensions to $C^1$ partially hyperbolic diffeomorphisms that are used in our arguments.

\subsection{Katok shadowing lemma for $C^1$ diffeomorphisms with a dominated splitting}\label{a.katok}

This section is devoted to the prof of Lemma~\ref{l.katok_shadowing_lemma}.

\begin{lemma}\label{l.katok_shadowing_lemma_bis}
Let $f:M\to M$ be a $C^1$ diffeomorphism and $\mu$ be an ergodic hyperbolic probability measure.
Assume that the Oseledets splitting $O^s \oplus O^u$ extends to a dominated splitting of the
tangent bundle over the support of $\mu$.
Then $\mu$ is the weak$^*$ limit of a sequence of invariant measures $\mu_n$ supported on the
orbits of hyperbolic periodic points $p_n$ such that the sequence $(p_n)_n$ converges to $\supp\mu$,
each $p_n$ has $\dim O^s$ contracting eigenvalues, and their stable manifolds have uniform size.
\end{lemma}

Let $E \oplus F$ denote the dominated splitting of that extends the Oseledets splitting of $f$ on the support of $\mu$.
Since the Lyapunov exponents of $\mu$ along the sub-bundle $E$ are all negative, there are  $a < 0$ and $N_s \ge 1$
such that
$$
\frac{1}{N_s} \int_M \log \|Df^{N_s}\mid_{E}\| \, d\mu < a.
$$
Then there exists some ergodic component $\mu_s$ of $\mu$ for the iterate $f^{N_s}$ such that
\begin{equation}\label{eq.Econtraction}
    \frac{1}{N_s} \int_M \log \|Df^{N_s}\mid_{E}\| \, d\mu_s < a.
\end{equation}
We claim that there exists a set $\Lambda^s \subset M$ with positive $\mu_s$-measure such that
\begin{equation}\label{eq.shyperbolictime}
    \prod_{i=0}^{n-1}\|Df^{N_s}\mid_{E(f^{iN_s}(x))}\| < e^{anN_s}.
\end{equation}
for any $x \in \Lambda^s$ and $n\geq 1$.
Indeed, suppose that for $\mu_s$-almost every $x$ there is $n_x \ge 1$ such that
$$
\prod_{i=0}^{n_x-1}\|Df^{N_s}\mid_{E(f^{iN_s}(x))}\|\geq e^{an_xN_s}.
$$
Then, still for $\mu_s$-almost every $x$,
$$
\limsup_n \frac{1}{n N_s}\sum_{i=0}^{n-1} \log\|Df^{N_s}\mid_{E(f^{iN_s}(x))}\|\geq a,
$$
which, in  of the Birkhoff ergodic theorem, contradicts \eqref{eq.Econtraction}.
This contradiction proves the claim \eqref{eq.shyperbolictime}.
The same argument proves that there are $b<0$, $N_u\ge 1$, an ergodic component $\mu_u$ of $\mu$
and a set $\Lambda^u \subset M$ with positive $\mu_u$-measure such that
\begin{equation}\label{eq.uhyperbolictime}
    \prod_{i=0}^{n-1}\|Df^{-N_u}\mid_{F(f^{-iN_u}(x))}\| < e^{bnN_u}
\end{equation}
for any $x\in \Lambda^u$ and any $n\geq 1$.

Since $\mu$ is ergodic for $f$, it has finitely many ergodic components for both $f^{N_s}$ and $f^{N_u}$.
Thus, the previous construction implies that both $\Lambda^s$ and $\Lambda^u$ have positive $\mu$-measure.
Using ergodicity once more, it follows that there is $m \ge 1$ such that $\Lambda_0=f^m(\Lambda^u) \cap \Lambda^s$
has positive $\mu$-measure.
Moreover, by subadditivity, every sufficiently large multiple $N \ge 1$ of $N_s N_u$ satisfies
\begin{equation}\label{eq.sbothhyperbolic}
\frac{1}{n}\sum_{i=0}^{n-1} \log\|Df^{N}\mid_{E(f^{iN}(x))}\| < tN_ua<0
\end{equation}
and
\begin{equation}\label{eq.ubothhyperbolictime}
\frac{1}{n}\sum_{i=0}^{n-1} \log\|Df^{-N}\mid_{F(f^{-iN}(x))}\| < \frac{tN_sb}{2}<0.
\end{equation}
for any $x\in \Lambda_0$ and $n \ge 1$. The reason we need to take $N$ sufficiently large is to compensate for the
fact that in \eqref{eq.ubothhyperbolictime} time is shifted by $m$ steps relative to \eqref{eq.uhyperbolictime}.

By Poincar\'e recurrence, for $\mu$-almost every $x\in \Lambda_0$ there is $q_x \ge 1$ arbitrarily large such that
$f^{q_x N}(x) \in \Lambda_0$ and is arbitrarily close to $x$.
Then $x$ satisfies \eqref{eq.sbothhyperbolic} for $n\in \{1, \dots, q_x\}$ and $f^{q_xN}(x)$ satisfies
\eqref{eq.ubothhyperbolictime} for $n \in \{1, \dots, q_x\}$.
This means that $\{f^{iN}(x): i=0, \dots, q_x\}$ is a \emph{quasi-hyperbolic string} in the sense of Liao~\cite{Lia89}:
this means that $E$ is forward contracting and $F$ is backward contracting along this orbit segment.
Then by the Liao~\cite{Lia89} shadowing lemma (see also Gan~\cite{Gan02}), there is a periodic point $p$
of period $q_x$ for $f^N$ close to $x$ and whose orbit shadows the pseudo-orbit $\{x, f^N(x),\cdots, f^{(q_x-1)N}(x)\}$.
Indeed, there exists $L>0$ depending only on $f^N$, $N_ua$ and $N_sb$, such that the
\begin{equation}\label{eq.distance_above}
d(f^{iN}(x),f^{iN}(p)) \le L d(x,f^{q_xN}(x)) \text{ for every } i=0, \dots, q_x.
\end{equation}
Since the distance on the right hand side may be chosen arbitrarily small, we may conclude that the inequalities
\eqref{eq.sbothhyperbolic} and \eqref{eq.ubothhyperbolictime} are inherited by the point $p$.
Observe also that, as $N$ has been fixed, up to increasing $L$ if necessary we may now assume that \eqref{eq.distance_above}
remains true for every iterate $f^j$ with $j=0, \dots, q_xN$.

By \cite{HPS77}, in the presence of a dominated splitting $E \oplus F$ there exist continuous families $\cF^E_{loc}(x)$ and $\cF^F_{loc}(x)$ of local sub-manifolds tangent to the invariant sub-bundles $E$
and $F$, respectively, which are locally invariant: there exists $r_0>0$ such that
\begin{equation}\label{eq.localinvariant}
f(\cF^E_{r_0}(x))\subset \cF^E_{loc}(f(x))
\text{ and }
f^{-1}(\cF^F_{r_0}(x))\subset \cF^F_{loc}(f^{-1}(x)),
\end{equation}
where $\cF^E_{r_0}(x)$ and $\cF^F_{r_0}(x)$ denote the $r_0$-neighborhoods of $x$ inside $\cF^E_{loc}(x)$
and $\cF^F_{loc}(x)$, respectively, relative to the respective leaf distances.
These are usually referred to as \emph{fake foliations}, see \cite{BW10,LVY13}.

The next statement was proved by Alves, Bonatti, Viana~\cite[Lemma~2.7]{ABV00},
see also Dolgopyat~\cite[Lemma 8.1]{Dol00}.

\begin{lemma}\label{l.stablemanifold}
Let $\lambda>0$ and $\epsilon\in(0,r_0)$. If $x\in M$ and $n \geq 1$ are such
that $\|(DF^j \mid_{E^{cs}})(x)\| \leq e^{-\lambda j}$ for $j=1, \dots, n$, then
$$
F^j(\cF^{cs}_{\epsilon}(x)) \subset \cF^{cs}_{r_j}(F^j(x))
$$
for every $0 \leq j \leq n$, where $r_j = \epsilon e^{-\lambda j/2}$.
\end{lemma}

Let $(p_n)_n \to x$ be a sequence of periodic points obtained in this way, for values $(q_n)_n$ of $q_x \ge 1$ going to
infinity and such that $f^{q_n N}(x)$ converges to $x$.
By Lemma~\ref{l.stablemanifold}, these periodic points have stable manifolds with size uniformly bounded from below
(the bound depends only on $f^N$ and $N_u a$). Clearly, we may take $x$ such that
$$
\frac{1}{k-1} \sum_{j=0}^{k-1} \delta_f^{j}(x) \to \mu
$$
as $k \to \infty$. Then, using the inequality \eqref{eq.distance_above} for each iterate $f^j$ with $j=0, \dots, q_nN$,
we get that
$$
\frac{1}{q_n N}\sum_{j=0}^{q_n N-1} \delta_{f^j(x)}
$$
converges to $\mu$ as $n\to\infty$. This completes the proof of Lemma~\ref{l.katok_shadowing_lemma}.

\subsection{Pesin stable manifold theorem for $C^1$ diffeomorphims with a dominated splitting}

The Pesin stable manifold theorem for non-uniformly hyperbolic systems is usually not true in the
$C^1$ category (see Pugh~\cite{Pug84}). The partial version we give here, for systems with a dominated
splitting, is probably known to the experts, but we could not find it in the literature.

Non-uniform hyperbolicity ensures that, for almost every point, a neighborhood of $x$ inside the fake leaf $\cF_{loc}^E(x)$ (respectively, $\cF_{loc}^F(x)$) is exponentially contracted
under forward (respectively, backward) iteration.
Note that these local Pesin stable and unstable manifolds turn out to vary continuously with the point $x$,
in this setting, except for their size which is only a measurable function of the point.

We formalize the conclusion in the next lemma.
Let $d^E_x(\cdot,\cdot)$ denote the leaf distance inside each $\cF^E_{loc}(x)$.

\begin{lemma}\label{l.stablemanifol}
Under the hypotheses of Lemma~\ref{l.katok_shadowing_lemma}, there are $\lambda <0$ and a full $\mu$-measure
invariant set $\Lambda$ such that for every $x\in \Lambda$, there are $C(x) >0$ and a neighborhood
$\cW^s_{loc}(x)$ of $x$ inside the  leaf $\cF^E_{loc}(x)$ satisfying:
\begin{enumerate}
    \item $f(\cW^s_{loc}(x))\subset \cW^s_{loc}(f(x))$ for every $x \in \Lambda$ and
    \item for any $y,z\in \cW^s_{loc}(x)$ and any $n\geq 0$,
    $$
    d^E_{f^n(x)}(f^n(y),f^n(z)) \leq C(x)e^{n\lambda} d^E_x(y,z)
    $$
    \end{enumerate}
for every $x\in\Lambda$. Moreover, $C(f(x))\leq C(x)\|Df\|$ for every $x\in\Lambda$,
and the map $x \mapsto \cW^s_{loc}(x)$ is measurable.
\end{lemma}

\begin{proof}
Let $\Lambda_0$ be a positive $\mu$-measure set of points  satisfying \eqref{eq.sbothhyperbolic} for $f^N$.
Then, by Lemma~\ref{l.stablemanifold}, there are $r>0$ and $\lambda_E<0$ such that
\begin{equation}\label{eq.contrai1}
d^E_{f^{nN}(x)}(f^{nN}(y), f^{nN}(z)) \leq e^{n\lambda_E} d^E_x(y,z)
\end{equation}
for any $y,z \in \cF^E_r(x)$ and $n\geq 0$.
Since $\|Df\|$ is uniformly bounded, it follows that there exists $N_E\ge 1$ such that
\begin{equation}\label{eq.contrai2}
d^E_{f^{m}(x)}(f^{m}(y), f^{m}(z)) \leq d^E_x(y,z)
\text{ for every } m \ge N_E.
\end{equation}
Define $\lambda = \lambda_E/N$ and take $\Lambda$ to be the (full $\mu$-measure) set of points whose forward
and backward orbits visit $\Lambda_0$ infinitely many times.
For any $x\in\Lambda_0$, let $n_x \ge 1$ be its first-return time to $\Lambda_0$.
It is no restriction to assume that $n_x \ge N$ for every $x \in \Lambda_0$: just reduce $\Lambda_0$
if necessary, observing that (by ergodicity) this does not affect $\Lambda$.
Define
$$
\cW^s_{loc}(f^i(x)) = f^{i}(\cF^E_r(x)) \text{ for } i = 0, \dots, n_x-1 \text{ and } x \in \Lambda_0.
$$
It is clear that $W^s_{loc}(y)$ depends measurably on $y \in \Lambda$.
The claim (1) in the statement is an immediate consequence of the definition, except possibly at the returns
$n_x$ to $\Lambda_0$, where it follows from \eqref{eq.contrai2} together with the fact that $n_x \ge N$.
Since $\|Df\|$ is uniformly bounded, property \eqref{eq.contrai1} also implies that for each $x \in \Lambda$
one can find $C(x)>1$ such that
\begin{equation}\label{eq.contrai3}
d^E_{f^{m}(x)}(f^{m}(y), f^{m}(z)) \leq  C(x) e^{n\lambda} d^E_x(y,z)
\end{equation}
for any $y,z \in \cF^E_r(x)$ and $n\geq 0$. This is claim (2) in the statement, and it is easy to check
that $C(x)$ may be chosen such that $C(f(x)) \le C(x) \|Df\|$.
\end{proof}

\subsection{Hyperbolic measures of $C^1$ partially hyperbolic diffeomorphisms}

Finally, we prove a $C^1$-version of a result of Ledrappier~\cite{Led84a}: the original
statement assumes that the derivative is H\"older, whereas in our statement the
diffeomorphism is taken to be partially hyperbolic.
For the statement, we must recall the definition of $u$-entropy for diffeomorphisms
that need not factor to an Anosov automorphism.

Let $f: M\to M$ be a $C^1$ partially hyperbolic diffeomorphism and $\mu$ be an ergodic
measure of $f$. A measurable partition $\xi$ of $M$ is said to be \emph{$\mu$-subordinate}
to the strong-unstable foliation $\cF^{uu}$ if for $\mu$-almost every $x$:
\begin{itemize}
    \item[(1)] $\xi(x)\subset \cF^{uu}(x)$ and it has uniformly small diameter inside $\cF^{uu}(x)$;
    \item[(2)] $\xi(x)$ contains an open neighborhood of $x$ inside the leaf $\cF^{uu}(x)$;
    \item[(3)] $\xi$ is an increasing partition, meaning that $f(\xi) \prec \xi$.
\end{itemize}
Then the \emph{$u$-entropy} of $\mu$ is defined by
$$
h_\mu(f,\cF^{uu})=H_\mu(\xi\mid f(\xi)).
$$
The definition does not depend on the choice of $\xi$ (see Ledrappier, Young~\cite[Lemma~3.1.2]{LeY85a}).

These notions go back to Ledrappier, Strelcyn~\cite{LeS82}, who proved that every
non-uniformly hyperbolic diffeomorphism with H\"older derivative admits measurable partitions subordinate to the corresponding Pesin unstable lamination.
In our present setting, partitions subordinate to the strong-unstable foliation of a
partially hyperbolic diffeomorphism were constructed by Yang~\cite[Lemma~3.2]{Yan16}.

\begin{proposition}\label{p.uentropyandentropy}
Let $f: M\to M$ be a $C^1$ partially hyperbolic diffeomorphism,
and $\mu$ be an ergodic measure of $f$ whose center Lyapunov exponents are all negative.
Then $h_\mu(f)=h_\mu(f,\cF^{uu})$.
\end{proposition}

Recall (Rokhlin~\cite{Rok67a}, see also \cite[Lemma~9.1.12]{FET16}) that the entropy of
$(f,\mu)$ relative to a finite partition $\cA$ can be defined as
$$
h_\mu(f,\cA)
= H_\mu\left(\cA\mid \bigvee_{i=1}^\infty f^i(\cA)\right)
= \lim_n H_\mu\left(\cA\mid \bigvee_{i=1}^n f^i(\cA)\right).
$$
We denote by $\cA^u$ the partition of $M$ whose elements are the intersections of
the elements of $\cA$ with local strong-unstable leaves.

\begin{lemma}\label{l.partitio_entropy}
Under the hypotheses of Proposition~\ref{p.uentropyandentropy}, if $\cA$ is a finite
partition with sufficiently small diameter then $h_\mu(f)=h_\mu(f,\cA)$ and,
up to zero $\mu$-measure,
\begin{enumerate}
\item $\bigvee_{i=0}^\infty f^i(\cA) = \bigvee_{i=0}^\infty f^i(\cA^u)$ and
\item $\bigvee_{i=-\infty}^\infty f^{i}(\cA)$ is the partition into points.
\end{enumerate}
\end{lemma}

\begin{proof}
As before, let $\cF^{uu}$ denote the strong-unstable foliation of $f$.
Moreover, let $\cF^{cs}_{loc}$ be a fake center-stable foliation,
that is, a locally invariant continuous family of local sub-manifolds
tangent to the center-stable sub-bundle $E^{cs}$.
Recall \eqref{eq.localinvariant} in the previous section.

Just as in the proof of Lemma~\ref{l.stablemanifol}, we can find $r>0$,
$\lambda<0$ and a positive $\mu$-measure set $\Lambda_0$ of points
whose stable manifolds have uniform size $r>0$, and uniform contraction
rate $e^{\lambda}$ under forward iteration by $f$.

Let $x$ and $y$ be any two nearby points in $M$, there is a unique
point of intersection the local strong-unstable leaf $\cF^{uu}_{loc}(y)$
and the local centre-stable leaf $\cF^{cs}_{r}(x)$, which we denote
as $z=[x,y]$. By local invariance, assuming $x$ and $y$ are sufficiently close to each other we also have that
\begin{equation}\label{eq.bracket_invariance}
f\left([x,y]\right) = \left[f(x),f(y)\right]
\text{ and }
f^{-1}\left([x,y]\right) = \left[f^{-1}(x),f^{-1}(y)\right]
\end{equation}
We assume that the diameter of the partition $\cA$ is small enough that
this will happen whenever $x$ and $y$ are in the same element of $\cA$.

By further reducing the diameter of $\cA$ if necessary, we may assume
that any two points in distinct elements of $\cA^u$ inside the same
element of $\cA$ cannot be mapped into the same element of $\cA^u$ by
either $f$ or its inverse.
More precisely, if $x$ and $y$ are such that $y\in \cA(x)$ and
$f(y)\in \cA(f(x))$. Then $y \in \cA^u(x)$ if and only if
$f(y) \in \cA^u(f(x))$ and similarly for the inverse $f^{-1}$.

Let us prove part (1) of the lemma.
It is clear that $\bigvee_{i=0}^\infty f^i(\cA)$ is coarser than $\bigvee_{i=0}^\infty f^i(\cA^u)$.
To prove the converse, we only need to show that
\begin{equation}\label{eq.pushforwardpartition}
\cA^u \prec \bigvee_{i\geq 0} f^i(\cA)
\end{equation}
up to a zero $\mu$-measure.
Indeed, \eqref{eq.pushforwardpartition} implies that $f^n(\cA^u)$ is
coarser than $\bigvee_{i=n}^\infty f^i(\cA)$ for any $n \ge 1$,
and so
$$
\bigvee_{n=0}^\infty f^n(\cA^u)
\prec \bigvee_{n=0}^\infty\left(\bigvee_{i=n}^\infty f^i(\cA)\right)
= \bigvee_{i=0}^\infty f^i(\cA).
$$

To prove \eqref{eq.pushforwardpartition}, suppose, by contradiction,
that there exists a positive $\mu$-measure set of points $x$,
and for each of them there exists
$$
y\in \left(\bigvee_{i\geq 0} f^i(\cA)\right) (x) \setminus \cA^u(x).
$$
Then the point $z=[x,y]$ is distinct from $x$ and, moreover, each
backward iterate $f^{-i}(y)$ is in the same element of $\cA$ as
the corresponding $f^{-i}(x)$.
Thus, we may us \eqref{eq.bracket_invariance} to conclude that
$$
f^{-i}(z)=[f^{-i}(x),f^{-i}(y)] \text{ for any $i\geq 0$.}
$$

By Poincar\'e recurrence, there is a sequence of times
$0<i_1<\cdots < i_j < \cdots <\infty$ such that
$f^{-i_j}(x)\in \Lambda_0$ for every $j$.
By further assuming that the diameter of $\cA$ is small,
it follows that
$$
f^{-i_j}(z)
\in \cF^{cs}_{r}(f^{-i_j}(x))
= W^s_{r}(f^{-i_j}(x)).
$$
Since the local stable manifolds have a uniform contraction rate,
we get that
$$
d^{cs}_x(x,z) \leq C_0 e^{i_j \lambda} \text{ for every } j.
$$
Making $j\to \infty$ we conclude that $x=z$, which is a contradiction.
This contradiction proves the claim in \eqref{eq.pushforwardpartition},
which completes the proof of part (1) of the lemma.

Finally, observe that \eqref{eq.pushforwardpartition} also yields
$$
\left(\bigvee_{-\infty}^\infty f^{i}(\cA)\right)(x)
\subset \left(\bigvee_{i\leq 0} f^{i}(\cA)\right)(x)\cap \cA^u(x).
$$
Since the strong-unstable foliation $\cF^{uu}$ is uniformly expanding,
the expression on the right hand side reduces to $\{x\}$, as long as
the diameter of $\cA$ is small enough. This proves part (2) of the lemma.
\end{proof}

\begin{proof}[Proof of Proposition~\ref{p.uentropyandentropy}]
Choose a finite partition $\cA$ with diameter small enough so that
Lemma~\ref{l.partitio_entropy} applies. By Yang~\cite[Proposition 3.1]{Yan16},
we may always choose $\cA$ in such a way that the boundary has small $\mu$-measure,
in the sense that there is $\lambda<1$ such that the $\mu$-measures of the
$\lambda^n$-neighborhoods of the boundary form a summable series.
Then, by Yang~\cite[Lemma 3.2]{Yan16}, the partition
$$
\xi = \bigvee_{i=0}^\infty f^i(\cA) = \bigvee_{i=0}^\infty f^i(\cA^u)
$$
in part (1) of Lemma~\ref{l.partitio_entropy} is $\mu$-subordinate.
Thus the lemma gives that
$$
\begin{aligned}
h_\mu(f)
& = h_\mu(f,\cA)
= H_\mu\left(\cA\mid \bigvee_{i=1}^\infty f^i(\cA)\right)
= H_\mu\left(\cA\mid \bigvee_{i=1}^\infty f^i(\cA^u)\right)\\
& = H_\mu\left(\cA \mid f(\xi)\right)
= H_\mu\left(\cA\vee f(\xi) \mid f(\xi)\right).
\end{aligned}
$$
Thus, we only need to show that $\cA \vee f(\xi)=\xi$ up to $\mu$-measure zero.

On the one hand, since $\cA$ is coarser than $\cA^u$,
it is clear that $\cA \vee f(\xi)$ is coarser than
$$
\cA^u \vee f(\xi)
= \cA^u \vee \left(\bigvee_{i=1}^\infty f^i(\cA^u)\right)=\xi.
$$
On the other hand, since $\xi$ is $\mu$-subordinate and has small diameter,
every element of $f(\xi)$ is contained in a local strong-unstable leaf.
Thus every element of $\cA \vee f(\xi)$ is contained in the intersection of
an element of $\cA$ with a local strong-unstable leaf.
In other words, $\cA \vee f(\xi)$ is finer than $\cA^u$, which implies that
$$
\cA\vee f(\xi)=(\cA \vee f(\xi)) \vee f(\xi)
$$
is finer than $\cA^u \vee f(\xi)=\xi$.
This proves that $\cA \vee f(\xi)=\xi$ up to $\mu$-measure zero, as we claimed.
\end{proof}

\bibliographystyle{alpha}
\bibliography{bib}

\newcommand{\etalchar}[1]{$^{#1}$}
\begin{thebibliography}{BRHRH{\etalchar{+}}08}

\bibitem[ABC05]{ABC05}
M.-C. Arnaud, C.~Bonatti, and S.~Crovisier.
\newblock Dynamiques symplectiques g\'{e}n\'{e}riques.
\newblock {\em Ergodic Theory Dynam. Systems}, 25:1401--1436, 2005.

\bibitem[ABV00]{ABV00}
J.~F. Alves, C.~Bonatti, and M.~Viana.
\newblock S{RB} measures for partially hyperbolic systems whose central
  direction is mostly expanding.
\newblock {\em Invent. Math.}, 140:351--398, 2000.

\bibitem[And10]{ANd10}
M.~Andersson.
\newblock Robust ergodic properties in partially hyperbolic dynamics.
\newblock {\em Trans. Amer. Math. Soc.}, 362:1831--1867, 2010.

\bibitem[AV10]{Extremal}
A.~Avila and M.~Viana.
\newblock Extremal {L}yapunov exponents: an invariance principle and
  applications.
\newblock {\em Invent. Math.}, 181:115--189, 2010.

\bibitem[AVW]{AVW2}
A.~Avila, M.~Viana, and A.~Wilkinson.
\newblock Absolute continuity, rigidity and {L}yapunov exponents {II}: systems
  with compact center leaves.
\newblock Preprint www.impa.br$\sim$viana 2019.

\bibitem[BB12]{BoB12}
J.~Bochi and C.~Bonatti.
\newblock Perturbation of the {L}yapunov spectra of periodic orbits.
\newblock {\em Proc. London Math. Soc.}, 105:1--48, 2012.

\bibitem[BC04]{BoCr04}
C.~Bonatti and S.~Crovisier.
\newblock R\'{e}currence et g\'{e}n\'{e}ricit\'{e}.
\newblock {\em Invent. Math.}, 158:33--104, 2004.

\bibitem[BCF18]{BCF18}
J.~Buzzi, S.~Crovisier, and T.~Fisher.
\newblock The entropy of $c^1$-diffeomorphisms without a dominated splitting.
\newblock {\em Trans. Amer. Math. Soc.}, 370:6685--6734, 2018.

\bibitem[BDP03]{BDP03}
C.~Bonatti, L.~J. D{\'\i}az, and E.~Pujals.
\newblock A ${C}^1$-generic dichotomy for diffeomorphisms: weak forms of
  hyperbolicity or infinitely many sinks or sources.
\newblock {\em Annals of Math.}, 158:355--418, 2003.

\bibitem[BDV05]{Beyond}
C.~Bonatti, L.~J. D{\'{\i}}az, and M.~Viana.
\newblock {\em Dynamics beyond uniform hyperbolicity}, volume 102 of {\em
  Encyclopaedia of Mathematical Sciences}.
\newblock Springer-Verlag, 2005.

\bibitem[BLY13]{BLY13}
C.~Bonatti, M.~Li, and D.~Yang.
\newblock On the existence of attractors.
\newblock {\em Transactions of the A. M. S}, 365:1369--1391, 2013.

\bibitem[Bow75]{Bow75a}
R.~Bowen.
\newblock {\em Equilibrium states and the ergodic theory of {A}nosov
  diffeomorphisms}, volume 470 of {\em Lect. Notes in Math.}
\newblock Springer Verlag, 1975.

\bibitem[BR75]{BR75}
R.~Bowen and D.~Ruelle.
\newblock The ergodic theory of {A}xiom {A} flows.
\newblock {\em Invent. Math.}, 29:181--202, 1975.

\bibitem[BRHRH{\etalchar{+}}08]{BHHTU}
K.~Burns, F.~Rodriguez-Hertz, M~A. Rodriguez-Hertz, A.~Talitskaya, and R.~Ures.
\newblock Density of accessibility for partially hyperbolic diffeomorphisms
  with one-dimensional center.
\newblock {\em Discrete Contin. Dyn. Syst.}, 22:75--88, 2008.

\bibitem[BV00]{BoV00}
C.~Bonatti and M.~Viana.
\newblock S{RB} measures for partially hyperbolic systems whose central
  direction is mostly contracting.
\newblock {\em Israel J. Math.}, 115:157--193, 2000.

\bibitem[BW10]{BW10}
K.~Burns and A.~Wilkinson.
\newblock On the ergodicity of partially hyperbolic systems.
\newblock {\em Annals of Math.}, 171:451--489, 2010.

\bibitem[CPZ20]{CPZ20}
V.~Climenhaga, Ya. Pesin, and A.~Zelerowicz.
\newblock Equilibrium measures for some partially hyperbolic systems.
\newblock {\em J. Mod. Dyn.}, 16:155--205, 2020.

\bibitem[Did03]{Did03}
Ph. Didier.
\newblock Stability of accessibility.
\newblock {\em Ergodic Theory Dynam. Systems}, 23(6):1717--1731, 2003.

\bibitem[Din70]{Din70}
E.~Dinaburg.
\newblock A correlation between topological entropy and metric entropy.
\newblock {\em Dokl. Akad. Nauk SSSR}, 190:19--22, 1970.

\bibitem[Din71]{Din71}
E.~Dinaburg.
\newblock A connection between various entropy characterizations of dynamical
  systems.
\newblock {\em Izv. Akad. Nauk SSSR Ser. Mat.}, 35:324--366, 1971.

\bibitem[Dol00]{Dol00}
D.~Dolgopyat.
\newblock On dynamics of mostly contracting diffeomorphisms.
\newblock {\em Comm. Math. Phys}, 213:181--201, 2000.

\bibitem[DVY16]{DVY16}
D.~Dolgopyat, M.~Viana, and J.~Yang.
\newblock Geometric and measure-theoretical structures of maps with mostly
  contracting center.
\newblock {\em Comm. Math. Phys.}, 341:991--1014, 2016.

\bibitem[Fra70]{Fra70}
J.~Franks.
\newblock Anosov diffeomorphisms.
\newblock In {\em Global {A}nalysis ({P}roc. {S}ympos. {P}ure {M}ath., {V}ol.
  {XIV}, {B}erkeley, {C}alif., 1968)}, pages 61--93. Amer. Math. Soc., 1970.

\bibitem[Gan02]{Gan02}
S.~Gan.
\newblock A generalized shadowing lemma.
\newblock {\em Discrete Contin. Dyn. Syst.}, 8:627--632, 2002.

\bibitem[GLVY]{GLVY}
S.~Gan, M.~Li, M.~Viana, and J.~Yang.
\newblock Partially volume expanding diffeomorphisms.
\newblock Preprint www.impa.br$\sim$viana 2020.

\bibitem[Goo71a]{Gdm71}
T.~Goodman.
\newblock Relating topological entropy and measure entropy.
\newblock {\em Bull. London Math. Soc.}, 3:176--180, 1971.

\bibitem[Goo71b]{Gdw71}
G.~Goodwin.
\newblock Optimal input signals for nonlinear-system identification.
\newblock {\em Proc. Inst. Elec. Engrs.}, 118:922--926, 1971.

\bibitem[Gou10]{Gou10}
N.~Gourmelon.
\newblock Generation of homoclinic tangencies by {$C^1$}-perturbations.
\newblock {\em Discrete Contin. Dyn. Syst.}, 26:1--42, 2010.

\bibitem[Gou16]{Gou16}
N.~Gourmelon.
\newblock A {F}ranks lemma that preserves invariant manifolds.
\newblock {\em Ergodic Theory Dynam. Systems}, 36:1167--1203, 2016.

\bibitem[HHW17]{HHW17}
H.~Hu, Y.~Hua, and W.~Wu.
\newblock Unstable entropies and variational principle for partially hyperbolic
  diffeomorphisms.
\newblock {\em Adv. Math.}, 321:31--68, 2017.

\bibitem[Hir89]{Hir90}
K.~Hiraide.
\newblock Expansive homeomorphisms with the pseudo-orbit tracing property of
  {$n$}-tori.
\newblock {\em J. Math. Soc. Japan}, 41:357--389, 1989.

\bibitem[HP14]{HaP14}
A.~Hammerlindl and R.~Potrie.
\newblock Pointwise partial hyperbolicity in three-dimensional nilmanifolds.
\newblock {\em J. Lond. Math. Soc.}, 89:853--875, 2014.

\bibitem[HPS77]{HPS77}
M.~Hirsch, C.~Pugh, and M.~Shub.
\newblock {\em Invariant manifolds}, volume 583 of {\em Lect. Notes in Math.}
\newblock Springer Verlag, 1977.

\bibitem[HWZ]{HWZ}
H.~Hu, W.~Wu, and Y.~Zhu.
\newblock Unstable pressure and u-equilibrium states for partially hyperbolic
  diffeomorphisms.
\newblock arXiv:1710.02816v1.

\bibitem[Kat80]{Kat80c}
A.~Katok.
\newblock Lyapunov exponents, entropy and periodic orbits for diffeomorphisms.
\newblock {\em Inst. Hautes \'{E}tudes Sci. Publ. Math.}, 51:137--173, 1980.

\bibitem[Led84]{Led84a}
F.~Ledrappier.
\newblock Propri{\'e}t{\'e}s ergodiques des mesures de {S}ina{\"\i}.
\newblock {\em Publ. Math. I.H.E.S.}, 59:163--188, 1984.

\bibitem[Lia89]{Lia89}
S.-T. Liao.
\newblock On {$(\eta,d)$}-contractible orbits of vector fields.
\newblock {\em Systems Sci. Math. Sci.}, 2:193--227, 1989.

\bibitem[LS82]{LeS82}
F.~Ledrappier and J.-M. Strelcyn.
\newblock A proof of the estimation from below in {P}esin's entropy formula.
\newblock {\em Ergodic Theory Dynam. Systems}, 2:203--219 (1983), 1982.

\bibitem[LVY13]{LVY13}
G.~Liao, M.~Viana, and J.~Yang.
\newblock The entropy conjecture for diffeomorphisms away from tangencies.
\newblock {\em J. Eur. Math. Soc. (JEMS)}, 15(6):2043--2060, 2013.

\bibitem[LW77]{LeW77}
F.~Ledrappier and P.~Walters.
\newblock A relativised variational principle for continuous transformations.
\newblock {\em J. London Math. Soc.}, 16:568--576, 1977.

\bibitem[LY85a]{LeY85a}
F.~Ledrappier and L.-S. Young.
\newblock The metric entropy of diffeomorphisms. {I}. {C}haracterization of
  measures satisfying {P}esin's entropy formula.
\newblock {\em Ann. of Math.}, 122:509--539, 1985.

\bibitem[LY85b]{LeY85b}
F.~Ledrappier and L.-S. Young.
\newblock The metric entropy of diffeomorphisms. {I}{I}. {R}elations between
  entropy, exponents and dimension.
\newblock {\em Ann. of Math.}, 122:540--574, 1985.

\bibitem[Ma{\~{n}}78]{Man78}
R.~Ma{\~{n}}{\'{e}}.
\newblock Contributions to the stability conjecture.
\newblock {\em Topology}, 17:383--396, 1978.

\bibitem[Mar70]{Mar70}
G.~A. Margulis.
\newblock Certain measures that are connected with {A}nosov flows on compact
  manifolds.
\newblock {\em Funkcional. Anal. i Prilo\v{z}en.}, 4:62--76, 1970.
\newblock In Russian.

\bibitem[New78]{New78a}
S.~Newhouse.
\newblock Topological entropy and {H}ausdorff dimension for area preserving
  diffeomorphisms of surfaces.
\newblock In {\em Dynamical systems, {V}ol. {III}---{W}arsaw}, pages 323--334.
  Ast\'{e}risque, No. 51. Soc. Math. France, 1978.

\bibitem[PdM82]{PdM82_EN}
J.~Palis and W.~de~Melo.
\newblock {\em Geometric theory of dynamical systems}.
\newblock Springer-Verlag, 1982.
\newblock An introduction, Translated from the Portuguese by A. K. Manning.

\bibitem[Pes77]{Pes77}
Ya.~B. Pesin.
\newblock Characteristic {L}yapunov exponents and smooth ergodic theory.
\newblock {\em Russian Math. Surveys}, 324:55--114, 1977.

\bibitem[Ply80]{Ply80}
R.~V. Plykin.
\newblock Hyperbolic attractors of diffeomorphisms.
\newblock {\em Usp. Math. Nauk}, 35:94--104, 1980.
\newblock English translation: Russ. Math. Survey, 35 (1980), no. 3, 109--121.

\bibitem[Pot15]{Pot15}
R.~Potrie.
\newblock Partial hyperbolicity and foliations in {$\mathbb{T}^3$}.
\newblock {\em J. Mod. Dyn.}, 9:81--121, 2015.

\bibitem[PS82]{PeS82}
Ya. Pesin and Ya. Sinai.
\newblock {G}ibbs measures for partially hyperbolic attractors.
\newblock {\em Ergod. Th. {\&} Dynam. Sys.}, 2:417--438, 1982.

\bibitem[PS00]{PS00a}
E.~Pujals and M.~Sambarino.
\newblock Homoclinic tangencies and hyperbolicity for surface diffeomorphisms.
\newblock {\em Annals of Math.}, 151:961--1023, 2000.

\bibitem[Pug84]{Pug84}
C.~Pugh.
\newblock The {$C\sp{1+\alpha }$} hypothesis in {P}esin theory.
\newblock {\em Publ. Math. IHES}, 59:143--161, 1984.

\bibitem[RHRHTU12]{HHTU12}
F.~Rodriguez-Hertz, M.~A. Rodriguez-Hertz, A.~Tahzibi, and R.~Ures.
\newblock Maximizing measures for partially hyperbolic systems with compact
  center leaves.
\newblock {\em Ergodic Theory Dynam. Systems}, 32:825--839, 2012.

\bibitem[Rok67]{Rok67a}
V.~A. Rokhlin.
\newblock Lectures on the entropy theory of measure-preserving transformations.
\newblock {\em Russ. Math. Surveys}, 22 -5:1--52, 1967.
\newblock Transl. from Uspekhi Mat. Nauk. 22 - 5 (1967), 3--56.

\bibitem[Rue76]{Rue76b}
D.~Ruelle.
\newblock A measure associated with {A}xiom {A} attractors.
\newblock {\em Amer. J. Math.}, 98:619--654, 1976.

\bibitem[Rue78]{Rue78}
D.~Ruelle.
\newblock An inequality for the entropy of differentiable maps.
\newblock {\em Bull. Braz. Math. Soc.}, 9:83--87, 1978.

\bibitem[Sin72]{Sin72}
Ya. Sinai.
\newblock Gibbs measures in ergodic theory.
\newblock {\em Russian Math. Surveys}, 27:21--69, 1972.

\bibitem[Sma67]{Sma67}
S.~Smale.
\newblock Differentiable dynamical systems.
\newblock {\em Bull. Am. Math. Soc.}, 73:747--817, 1967.

\bibitem[SX09]{SaX09}
R.~Saghin and Z.~Xia.
\newblock Geometric expansion, {L}yapunov exponents and foliations.
\newblock {\em Ann. Inst. H. Poincar{\'e} Anal. Non Lin\'eaire}, 26:689--704,
  2009.

\bibitem[TY19]{TaY19}
A.~Tahzibi and J.~Yang.
\newblock Invariance principle and rigidity of high entropy measures.
\newblock {\em Trans. Amer. Math. Soc.}, 371:1231--1251, 2019.

\bibitem[Ure12]{Ure12}
R.~Ures.
\newblock Intrinsic ergodicity of partially hyperbolic diffeomorphisms with a
  hyperbolic linear part.
\newblock {\em Proc. Amer. Math. Soc.}, 140:1973--1985, 2012.

\bibitem[UVY]{UVY}
R.~Ures, M.~Viana, and J.~Yang.
\newblock Maximal measures of diffeomorphisms on circle fiber bundles.
\newblock {\em J. London Math. Soc.}
\newblock Preprint www.impa.br$\sim$viana 2019.

\bibitem[UVYY]{UVYY2}
R.~Ures, M.~Viana, F.~Yan, and J.~Yang.
\newblock Fast loss of memory for measures of maximal $u$-entropy of maps with
  $c$-mostly contracting center.
\newblock To appear.

\bibitem[Via08]{Almost}
M.~Viana.
\newblock Almost all cocycles over any hyperbolic system have nonvanishing
  {L}yapunov exponents.
\newblock {\em Ann. of Math.}, 167:643--680, 2008.

\bibitem[VO16]{FET16}
M.~Viana and K.~Oliveira.
\newblock {\em Foundations of ergodic theory}, volume 151 of {\em Cambridge
  Studies in Advanced Mathematics}.
\newblock Cambridge University Press, 2016.

\bibitem[VY13]{ViY13}
M.~Viana and J.~Yang.
\newblock Physical measures and absolute continuity for one-dimensional center
  direction.
\newblock {\em Ann. Inst. H. Poincar{\'e} Anal. Non Lin\'eaire}, 30:845--877,
  2013.

\bibitem[VY17]{ViY17}
M.~Viana and J.~Yang.
\newblock Measure-theoretical properties of center foliations.
\newblock In {\em Modern theory of dynamical systems}, volume 692 of {\em
  Contemp. Math.}, pages 291--320. Amer. Math. Soc., 2017.

\bibitem[VY19]{ViY19}
M.~Viana and J.~Yang.
\newblock Continuity of {L}yapunov exponents in the {$C^0$} topology.
\newblock {\em Israel J. Math.}, 229:461--485, 2019.

\bibitem[Wal75]{Wal75}
P.~Walters.
\newblock A variational principle for the pressure of continuous
  transformations.
\newblock {\em Amer. J. Math.}, 97:937--971, 1975.

\bibitem[Wen02]{Wen02}
L.~Wen.
\newblock Homoclinic tangencies and dominated splittings.
\newblock {\em Nonlinearity}, 15:1445--1469, 2002.

\bibitem[Yan]{Yan16}
J.~Yang.
\newblock Entropy along expanding foliations.
\newblock arXiv:1601.05504.

\bibitem[Yan08]{Yang-tese}
J.~Yang.
\newblock {\em {$C^1$} dynamics far from tangencies}.
\newblock PhD thesis, IMPA, 2008.
\newblock www.preprint.impa.br.

\end{thebibliography}

\end{document}